\documentclass {amsart}

\usepackage{amsfonts}
\usepackage{amssymb}

\newtheorem{thm}{Theorem}[section]
\newtheorem{cor}[thm]{Corollary}
\newtheorem{lem}[thm]{Lemma}

\theoremstyle{definition}

\newtheorem{exmp}{Example}[section]
\newtheorem{prob}{Problem}[section]

\theoremstyle{remark}
\newtheorem{rem}{Remark}[section]

\numberwithin{equation}{section}

\newcommand{\C}{\mathbb{C}}

\newcommand{\comp}{\operatorname{comp}}
\newcommand{\cD}{\mathcal{D}}
\newcommand{\Diff}{\mathrm{Diff}}

\newcommand{\cF}{\mathcal{F}}

\newcommand{\iG}{\mathit{\Gamma}}
\newcommand{\cH}{\mathcal{H}}
\renewcommand{\i}[1]{\raisebox{1pt}{$\stackrel{\scriptscriptstyle \circ}{#1}$}}

\newcommand{\loc}{\mathrm{loc}}

\newcommand{\N}{\mathbb{N}}

\newcommand{\iPhi}{\mathit{\Phi}}
\newcommand{\iPi}{\mathit{\Pi}}
\newcommand{\R}{\mathbb{R}}

\newcommand{\Sol}{\operatorname{\mathcal{S}}}
\newcommand{\supp}{\operatorname{supp}}
\renewcommand{\t}[1]{\tilde{#1}}

\begin{document}

\title[Mixed Problems with a Parameter]
      {Mixed Problems with a Parameter}


\author{A. Shlapunov}

\address[Alexander Shlapunov]
        {Krasnoyarsk State University
                                                 \\
         pr. Svobodnyi 79
                                                 \\
         660041 Krasnoyarsk
                                                 \\
         Russia}

\email{shlapuno@lan.krasu.ru}


\author{N. Tarkhanov}

\address[Nikolai Tarkhanov]
        {Universit\"{a}t Potsdam
                                                 \\
         Institut f\"{u}r Mathematik
                                                 \\
         Postfach 60 15 53
                                                 \\
         14415 Potsdam
                                                 \\
         Germany}

\email{tarkhanov@math.uni-potsdam.de}







\begin{abstract}
Let
   $X$ be a smooth $n\,$-dimensional manifold and
   $D$ be an open connected set in $X$ with smooth boundary $\partial D$.
Perturbing the Cauchy problem for an elliptic system $Au = f$ in $D$ with data
on a closed set $\iG \subset \partial D$ we obtain a family of mixed problems
depending on a small parameter $\varepsilon > 0$.
Although the mixed problems are subject to a non-coercive boundary condition
on $\partial D \setminus \iG$ in general,
   each of them is uniquely solvable in an appropriate Hilbert space
   $\cD_{T}$
and
   the corresponding family $\{ u_{\varepsilon} \}$ of solutions approximates
   the solution of the Cauchy problem in $\cD_{T}$ whenever the solution
   exists.
We also prove that the existence of a solution to the Cauchy problem in
$\cD_{T}$ is equivalent to the boundedness of the family
$\{ u_{\varepsilon} \}$.
We thus derive a solvability condition for the Cauchy problem and an effective
method of constructing its solution.
Examples for Dirac operators in the Euclidean space $\R^n$ are considered.
In the latter case we obtain a family of mixed boundary problems for the
Helmholtz equation.
\end{abstract}

\maketitle



\tableofcontents


\section*{Introduction}
\label{s.Int}

This paper is based on the following simple observation.
Consider an operator equation
   $Tu = f$
with a bounded operator $T : H \to \t{H}$ in Hilbert spaces.
If there is a $u \in H$ satisfying $Tu = f$ then $f$ is orthogonal to the
null-space of the adjoint operator $T^{\ast}$ in $\t{H}$.
On the other hand, for $f \in (\ker T^{\ast})^{\perp}$ the equation $Tu = f$
is obviously equivalent to $T^{\ast} T u = T^{\ast} f$.
The latter need not have any solution, however, the slightly perturbed equation
   $T^{\ast} T u + \varepsilon u = v$
is uniquely solvable for any $v \in H$, provided that $\varepsilon > 0$.
Note that the solution of the equation can be effectively constructed, for the
operator $T^{\ast} T + \varepsilon$ is positive definite.
We thus get a family
$$
u_{\varepsilon} = (T^{\ast} T + \varepsilon)^{-1} T^{\ast} f
$$
in $H$, whose limit is a good candidate for the solution of $Tu = f$ that is
orthogonal to the null-space of $T$.
Indeed, if $v \in H$ satisfies $Tv = 0$ then by Lemma 12.1.25 of \cite{Tark36}
we get
\begin{eqnarray*}
   (u_{\varepsilon}, v)_{H}
 & = &
   (f, T (T^{\ast} T + \varepsilon)^{-1} v)_{\t{H}}
                                                 \\
 & = &
   (f, (T T^{\ast} + \varepsilon)^{-1} T v)_{\t{H}}
                                                 \\
 & = &
   0,
\end{eqnarray*}
as desired.
If $f = Tu$ for some $u \in H$, then
$
   u_{\varepsilon}
 = u - \varepsilon (T^{\ast} T + \varepsilon)^{-1} u
$
is obviously bounded in $H$.

Conversely, if the norm $\| u_{\varepsilon} \|_{H}$ is bounded uniformly in
$\varepsilon \ll 1$ then $u_{\varepsilon}$ converges for $\varepsilon \searrow
0$ to the only solution $u \in H$ of $Tu = f$ that is orthogonal to $\ker T$.

In this way we derive a solvability condition and an approximate solution to
the equation $Tu = f$ in $H$.
We refer the reader to Section 12.1.5 of \cite{Tark36} for an extremal property
of $u_{\varepsilon}$.

When applying the approach in the study of the Cauchy problem for solutions of
an elliptic equation $Au = f$, one needs to complete it by refined analysis.
By the above, the calculus of the Cauchy problems which are ill-posed by the
very nature can be elaborated in the framework of the calculus of operators
$T^{\ast} T + \varepsilon I$ depending on a parameter $\varepsilon > 0$.
In order to avoid sophisticated adjoint operators one uses $L^{2}\,$-scalar
products which necessarily leads to unbounded closed operators with dense
domains.
Hence, it requires much more efforts to make use of the construction described
above.

The operator $T$ is given the domain consisting of those functions $u$ in $D$
which are square integrable along with $Au$ and whose Cauchy data with respect
to $A$ vanish on a closed set $\iG \subset \partial D$.
Then the domain of the adjoint operator $T^{\ast}$ is proved to consist of
square integrable functions $g$ on $D$, such that the Cauchy data of $g$ with
respect to $A^{\ast}$ vanish in the complement of $\iG$.
It follows that the natural domain of the Laplacian $T^{\ast} T$ is a subspace
of square integrable functions $u$ on $D$, such that the Cauchy data of $u$
with respect to $A$ vanish on $\iG$ and the Cauchy data of $Tu$ with respect
to $A^{\ast}$ vanish on $\partial D \setminus \iG$.
This gives rise to a mixed boundary value problem for the elliptic operator
$A^{\ast} A$ in $D$ similar to the classical Zaremba problem \cite{ZareX}.

Our paper demonstrates rather strikingly that the calculus of Cauchy problems
for solutions of elliptic equations just amounts to the calculus of mixed
boundary value problems for elliptic equations with a parameter, cf.
   \cite{SimaX}.
While this observation seems to be of purely mathematical interest, the
explicit solutions we construct by the classical Fourier method may be of
practical importance in applications.

\section{Preliminaries}
\label{s.2}

Let $X$ be a $C^{\infty}$ manifold of dimension $n$ with a smooth boundary
$\partial X$.
We tacitly assume that it is embedded into a smooth closed manifold $\t{X}$ of
the same dimension.

For any smooth $\C\,$-vector bundles $E$ and
                                     $F$
over $X$, we write
   $\Diff^{m} (X; E, F)$
for the space of all linear partial differential operators of order $\leq m$
between sections of $E$ and $F$.

We denote by $E^{\ast}$ the conjugate bundle of $E$.
Any Hermitian metric $(.,.)_{x}$ on $E$ gives rise to a sesquilinear bundle
isomorphism
   $\ast_{E} \! : E \to E^{\ast}$
by the equality
   $\langle \ast_{E} v, u \rangle_{x} = (u,v)_{x}$
for all sections $u$ and $v$ of $E$.

Pick a volume form $dx$ on $X$, thus identifying dual and conjugate bundles.
For $A \in \Diff^{m} (X; E, F)$, denote
   by $A' \in \Diff^{m} (X; F^{\ast}, E^{\ast})$ the transposed operator and
   by $A^{\ast} \in \Diff^{m} (X; F, E)$ the formal adjoint operator.
We have
   $A^{\ast} = \ast_{E}^{-1} A' \ast_{F}$,
cf. \cite[4.1.4]{Tark36} and elsewhere.

For an open set $O \subset X$, we write
   $L^{2} (O, E)$
for the Hilbert space of all measurable sections of $E$ over $O$ with a finite
norm
   $(u, u)_{L^{2} (O, E)} = \int_{O} (u,u)_{x} dx$.
We also denote by $H^{s} (O, E)$ the Sobolev space of distribution sections of
$E$ over $O$, whose weak derivatives up to order $s$ belong to $L^{2} (O, E)$.

Given any open set $O$ in $X^{\circ}$, the interior of $X$, we let
   $\Sol_{A} (O)$
stand for the space of weak solutions to the equation $Au = 0$ in $O$.
Obviously, the subspace of $H^{s} (O, E)$ consisting of all weak solutions to
$Au = 0$ is closed.

Write $\sigma^{m} (A)$ for the principal homogeneous symbol of the operator
$A$,
   $\sigma^{m} (A)$ living on the (real) cotangent bundle $T^{\ast} X$ of $X$.
From now on we assume that $\sigma^{m} (A)$ is injective away from the zero
section of $T^{\ast} X$.
Hence it follows that the Laplacian $A^{\ast} A$ is an elliptic differential
operator of order $2m$ on $X$.

If the dimensions of $E$ and $F$ are equal then $A$ is elliptic, too.
Otherwise we will call it overdetermined elliptic operator.

We can assume without restriction of generality that $A$ is included into a
compatibility complex of differential operators
   $A^{i} \in \Diff^{m_{i}} (X; E^{i}, E^{i+1})$
over $X$, where $i = 0, 1, \ldots, N$ and
                $A^{0} = A$.
This complex is elliptic in a natural way,
   see for instance \cite[4.1.2]{Tark35}).
If $A$ is elliptic then the compatibility complex is trivial, i.e.,
   $A^{i} = 0$ for all $i > 0$.

Let
   $D$ be a relatively compact domain in $X^{\circ}$ with a smooth boundary
   $\partial D$.
For $u \in L^{2} (D, E)$ we always regard $Au$ as a distribution section of
$F$ over $D$.

A large class of operators $A$ possess the following property which is usually
referred to as Unique Continuation Property,
\begin{description}
   \item [ $(U)_{s}$ ]
Given any domain $D \subset X^{\circ}$,
   if $u \in \Sol_{A} (D)$ vanishes on a non-empty open subset of $D$
then $u \equiv 0$ in all of $D$.
\end{description}

This property implies in particular the existence of a left fundamental
solution for $A$ in the interior of $X$.

Consider the Hermitian form
$$
D (u,v) = (u,v)_{L^{2} (D,E)} +(Au,Av)_{L^{2} (D,F)}
$$
on the space $C^\infty (\overline{D},E)$ of all smooth sections of $E$ over
the closure of $D$.
The functional
   $D (u) = \sqrt{D (u,u)}$
is usually called the Graph Norm related to the unbounded operator
   $A : L^{2} (D,E) \to L^{2} (D,F)$.
Write $\cD_{A}$ for the completion of $C^{\infty} (\overline{D}, E)$ with
respect to $D (\cdot)$.
Then $\cD_{A}$ is a Hilbert space with the scalar product $D (.,.)$, and
$A$ maps $\cD_{A}$ continuously to $L^{2} (D,F)$.

Note that if $A = \nabla$ is the gradient operator in $\R^n$ then
   $\cD_{A} = H^{1} (D)$.
Let us clarify what kind elements are in this space in the general case.

To this end  we fix a Dirichlet system
   $B_{j}$,  $j = 0, 1, \ldots, m-1$,
of order $m-1$ on $\partial D$.
More precisely, each $B_{j}$ is a differential operator of type $E \to F_{j}$
and order $m_{j} \leq m-1$ in a neighbourhood $U$ of $\partial D$, where
$m_{i} \neq m_{j}$ for $i \neq j$.
Moreover, the symbols $\sigma^{m_{j}} (B_{j})$,
   if restricted to the conormal bundle of $\partial D$,
have ranks equal to the dimensions of $F_{j}$.
Set
$$
t (u) = \oplus_{j=0}^{m-1} B_{j} u
$$
for $u \in H^{m} (D,E)$.

For $s > 0$ we denote by
   $H^{-s} (\partial D, F_{j})$
the dual of the space $H^{s} (\partial D, F_{j})$ with respect to the pairing
in $L^{2} (\partial D, F_{j})$.

\begin{lem}
\label{l.HDA}
For every $u \in \cD_{A}$, we have
   $u \in H^{m}_{\loc} (D,E)$.
Moreover  $t (u)$ has weak boundary values on $\partial D$ belonging to
   $\oplus_{j=0}^{m-1} H^{-m_{j}-1/2} (\partial D, F_{j})$.
\end{lem}

\begin{proof}
Fix an element $u \in \cD_{A}$.
Since $A$ is elliptic we deduce from $Au \in L^{2} (D,F)$ that
   $u \in H^{m}_{\loc} (D,E)$.

As usual, we denote by
   $H^{-m} (D,E)$
the completion of $C^{\infty} (\overline{D},E)$ with respect to the norm
$$
|u|_{-m} = \sup_{v \in C^{\infty} (\overline{D},E) \atop
                 t (v) = 0}
           \frac{|(u,v)_{L^{2} (D,E)}|}{\| v \|_{H^{m} (D,E)}}.
$$
Then we easily verify that $A^{\ast}$ extends to a map of $L^{2} (D,F)$ to
                                                          $H^{-m} (D,E)$,
more explicitly,
$$
(A^{\ast} f, v) := (f, Av)_{L^{2} (D,F)}
$$
for each $f \in L^{2} (D,F)$ and
         $v \in \i{H}{}^{m} (D,E)$.

By the very definition, the distribution $A^{\ast} f$ is always orthogonal under
the pairing in $L^{2} (D,E)$ to the null-space of the Dirichlet problem for
$A^{\ast} A$.
Therefore, for every $f \in L^{2} (D,F)$ there exists a section
   $Gf \in H^{m} (D,E)$
satisfying
   $A^{\ast} A\, Gf = A^{\ast} f$ in $D$ and
   $t (Gf) = 0$ on $\partial D$,
see for instance \cite{SchuShlaTark2}.
Any $u \in \cD_{A}$ can be thus presented in the form
$$
u = G\, Au + (u - G\, Au).
$$
By the construction, we get
   $G\, Au \in \i{H}{}^{m} (D,E)$ and
   $u - G\, Au \in \cD_{A} \cap \Sol_{A^{\ast} A} (D)$.
As $u - G\, Au \in L^{2} (D,E)$ is of finite order growth near $\partial D$,
we conclude by Lemma 9.4.4 of \cite{Tark36} that
   $t (u - G\, Au)$
has weak boundary values on $\partial D$  belonging to
   $\oplus_{j=0}^{m-1} H^{-m_{j}-1/2} (\partial D, F_{j})$.

As
   $t (G\, Au) \in \oplus_{j=0}^{m-1} H^{m-m_{j}-1/2} (\partial D, F_{j})$
vanishes on the boundary even in the usual sense for Sobolev spaces,
   the proof is complete.
\end{proof}

Let
   $\{ C_{j} \}_{j=0}^{m-1}$
be the adjoint Dirichlet system for $\{ B_{j} \}_{j=0}^{m-1}$ with respect to
the Green formula for $A$
   (see for instance \cite[Remark 9.2.6]{Tark36}).
For $g \in H^{m} (D, F)$, we set
$$
n (g) = \oplus_{j=0}^{m-1} C_{j} g.
$$

Suppose $\iG$ is a closed subset of $\partial D$.
The cases $\iG = \emptyset$ and
          $\iG = \partial D$
are permitted, too.
We write $\iG^{\circ}$ for the interior of $\iG$ in the relative topology of
$\partial D$.

Given any $u \in L^{2} (D,E)$ with $Au \in L^2 (D,F)$, we say that $t (u) = 0$
on the set $\iG$ if
\begin{equation}
\label{eq.def.t}
\int _{D} \left( (Au, g)_{x} - (u, A^{\ast} g)_{x} \right) dx = 0
\end{equation}
for all sections $g \in C^{\infty} (\overline{D}, F)$ satisfying $n (g) = 0$
on $\partial D \setminus \iG^{\circ}$.

\begin{lem}
\label{l.HDA.t.weakly}
If
   $u \in \cD_{A}$ and
   $t (u) = 0$ on $\iG$
then
   $u \in H^{m}_{\loc} (D \cup \iG^{\circ}, E)$.
\end{lem}

In particular, $t (u)$ has zero boundary values on $\iG^{\circ}$ in the usual
sense of Sobolev spaces.

\begin{proof}
The case $\iG = \emptyset$ has been already treated in Lemma \ref{l.HDA}.
Assume that $\iG$ is non-empty.

Choose a smooth real-valued function $\varrho$ on $X$ with the property that
\begin{equation}
\label{eq.rho}
D = \{ x \in X :\ \varrho (x) < 0 \}
\end{equation}
and $\nabla \varrho (x) \neq 0$ for all $x \in \partial D$.
Set
$
D_{\varepsilon} = \{ x \in X :\ \varrho (x) < \varepsilon \},
$
then
   $D_{-\varepsilon} \Subset D \Subset D_{\varepsilon}$ for all sufficiently
   small $\varepsilon > 0$,
and
   the boundary of $D_{\pm \varepsilon}$ is as smooth as the boundary of $D$.

We first show that the weak boundary values of $t (u)$ vanish on $\iG$ in the
sense that
$$
   \lim_{\varepsilon \to 0+}
   \int_{\partial D_{-\varepsilon}}
   \sum_{j=0}^{m-1} (B_{j} u, g_{j})_{x} ds
 = 0
$$
for all
   $g_{j} \in C^{\infty} (U, F_{j})$, $j = 0, 1, \ldots, m-1$,
satisfying $(\supp g_j) \cap \partial D \subset \iG$.
To this end, choose a function $g \in C^{\infty} (\overline{D}, F)$, such
that
   $n (g) = \oplus_{j=0}^{m-1} g_j$
on $\partial D$,
   cf. Lemma 9.3.5 in \cite{Tark36}.
Since $u \in L^{2} (D,E)$ and
      $Au \in L^{2} (D,F)$,
we obtain by the Green formula
\begin{eqnarray*}
   \lim _{\varepsilon \to 0+}
   \int_{\partial D_{-\varepsilon}}
   \sum_{j=0}^{m-1}
   (B_{j} u, g_{j})_x ds
 & = &
   \lim_{\varepsilon \to 0+}
   \int_{D_{-\varepsilon}}
   \left( (Au,g)_{x} - (u,A^{\ast} g)_{x} \right) dx
                                                 \\
 & = &
   \int _{D} \left( (Au,g)_x - (u,A^{\ast} g)_x \right) dx
                                                 \\[.4cm]
 & = &
   0
\end{eqnarray*}
because
   $t (u) = 0$ on $\iG$ in the sense of (\ref{eq.def.t})
and
   $g \in C^{\infty} (\overline{D}, F)$ satisfies $n (g) = 0$ on
   $\partial D \setminus \iG^{\circ}$.

We thus have
   $A^{\ast} A u \in H^{-m} (D,E)$
and
   the weak boundary values of $t (u)$ vanish on $\iG$.
As
   $A^{\ast} A$ is an elliptic operator of order $2m$
and
   $u \mapsto t (u)$ is a Dirichlet system of order $m-1$,
we conclude
   using the local regularity theorem for solutions of the Dirichlet problem
   for $A^{\ast} A$
that
   $u \in H^{m}_{\loc} (D \cup \iG^{\circ})$
(see for instance Theorem 9.3.17 of \cite{Tark36}), as desired.
\end{proof}

The proof actually shows that for sections $u \in L^{2} (D,E)$ with
                                           $Au \in L^{2} (D,F)$
the equality (\ref{eq.def.t}) just amounts to saying that the weak boundary
values of $t (u)$ vanish on $\iG^{\circ}$.

Let $\cD_{T}$ stand for the completion of the space of all sections $u$ in
$C^{\infty} (\overline{D},E)$, satisfying $t (u) = 0$ on $\iG$, with respect
to the norm $u \mapsto D (u)$.
By the very definition, $\cD_{T}$ is a closed subspace in $\cD_{A}$, and
it is a Hilbert space itself with the induced Hilbert structure.

It is well known that if $\iG$ is the whole boundary then
$
\cD_{T} = \i{H}{}^m (D,E),
$
the closure of $C^{\infty}_{\comp} (D,E)$ in $H^{m} (D,E)$.

\begin{lem}
\label{l.HD}
If  $u \in \cD_{T}$
then $t (u) = 0$ on $\iG$ in the sense of (\ref{eq.def.t}).
\end{lem}

\begin{proof}
If $u \in \cD_{T}$ then there exists a sequence  $\{ u_k \}_{k \in \N}$ in
$C^{\infty} (\overline{D},E)$ satisfying $t (u_{k}) = 0$ on $\iG$, such that
$$
\lim_{k \to \infty} D (u_k - u) = 0.
$$
Hence
\begin{eqnarray*}
   \int_{D} \left( (Au,g)_x - (u,A^{\ast} g)_x \right)dx
 & = &
   \lim_{k \to \infty}
   \int_{D} \left( (A u_k,g)_x - (u_k,A^{\ast} g)_x \right) dx
                                                 \\
 & = &
   \lim_{k \to \infty}
   \int_{\partial D}
   \sum_{j=0}^{m-1}
   (B_j u_k, C_j g)_x ds
                                                 \\[.3cm]
 & = &
   0
\end{eqnarray*}
for all $g \in C^{\infty} (\overline{D}, F)$ satisfying $n (g) = 0$ on
$\partial D \setminus \iG^{\circ}$, because $t (u_k) = 0$ on $\iG$.
Therefore, $t (u) = 0$ on $\partial \iG$.
\end{proof}

We are now in a position to characterise the space $\cD_{T}$ in a much more
convenient way.

\begin{thm}
\label{t.HD}
As defined above, $\cD_{T}$ is a closed subspace of $\cD_{A}$ consisting of
all $u \in \cD_{A}$ satisfying $t (u) = 0$ on $\iG$.
\end{thm}

\begin{proof}
Write $H$ for the subspace of $\cD_{A}$ consisting of all $u \in \cD_{A}$
satisfying $t (u) = 0$ on $\iG$.
It is easy to see that $H$ is a closed subspace of $\cD_{A}$.
Lemma \ref{l.HD} states that $\cD_{T}$ is a subspace of $H$.
Since $\cD_{T}$ is complete by the very definition, we shall have established
the theorem if we prove that the orthogonal complement $\cD_{T}^{\perp}$ of
$\cD_{T}$ in $H$ is zero.

To this end, pick a section $u \in H$ satisfying
   $D (u, v) = 0$
for all $v \in C^{\infty} (\overline{D},E)$, such that $t (v) = 0$ on $\iG$.
If moreover $v$ fulfills $n (Av) = 0$ on $\partial D \setminus \iG^{\circ}$
then we readily get
\begin{equation}
\label{eq.ort.Ct}
(u, (A^{\ast} A + I) v) _{L^2 (D,E)} = 0,
\end{equation}
which is due to (\ref{eq.def.t}).

We now observe that every $w \in C^{\infty} (\overline{D},E)$ can be
approximated in the $L^{2} (D,E)\,$-norm by sections of the form
$(A^{\ast} A + 1) v$, where $v \in C^{\infty} (\overline{D},E)$ satisfies
   $t (v) = 0$ on $\iG$
and
   $n (Av) = 0$ on $\partial D \setminus \iG^{\circ}$.
This latter is a consequence of the fact that the unbounded operator
$T^{\ast} T + 1$ in $L^{2} (D,E)$ with domain $\cD_{T^{\ast} T}$ is positive,
and so invertible,
   see \S~\ref{s.Ap} below.
We thus deduce from (\ref{eq.ort.Ct}) that $u = 0$.
It follows that $\cD_{T}^{\perp} = \{ 0 \}$, as desired.
\end{proof}

\section{The Cauchy problem}
\label{s.TCp}
\setcounter{equation}{0}

A rough formulation of the Cauchy problem for the operator $A$ in the domain
$D$ reads as follows:
Given any sections
   $f$ of $F$ over $D$ and
   $u_{0}$ of $\oplus_{j=0}^{m-1} F_{j}$ over $\iG$,
find a section $u$ of $E$ over $D$, such that
   $Au = f$ in $D$ and
   $t (u)$ has suitable limit values on $\iG$ coinciding with $u_0$.

Note that some regularity of $u$ up to $\iG$ is needed for $t (u)$ to possess
limit values on $\iG$.
Moreover, we are going to use Hilbert space methods for the study of the Cauchy
problem.
Hence the space $\cD_{A}$ seems to be a natural choice for posing the problem.

What is still lacking is a proper function space $B (\iG)$ for the Cauchy data
$u_{0}$ on $\iG$.
It is not difficult to introduce such a space in the case where $\iG$ is the
entire boundary, namely
$$
B (\partial D) = \cD_{A} / \i{H}{}^{m} (D,E).
$$
By Lemma \ref{l.HDA}, this quotient space can be specified within
   $\oplus_{j=0}^{m-1} H^{-m_{j}-1/2} (\partial D, F_j)$
under $t$, although the norm of the former is essentially stronger than the
norm of the latter.

Theorem \ref{t.HD} suggests us to set
$$
B (\iG) = \frac{\cD_{A}}{\cD_{T}}
$$
in general.
Using the approach of \cite[Ch.~1]{Tark36} one can specify $B (\iG)$ within
   $\oplus_{j=0}^{m-1} H^{-m_{j}-1/2} (\iG, F_j)$
under $t$.
Of course, it is difficult to explicitly describe the elements of $B (\iG)$.
However, for applications it suffices to know that there is a natural
embedding
$$
   \oplus_{j=0}^{m-1} H^{m-m_{j}-1/2} (\iG, F_j)
 \hookrightarrow
   B (\iG).
$$

Using the spaces $B (\iG)$ allows one to reduce the Cauchy problem with
non-zero Cauchy data on $\iG$ to the Cauchy problem with homogeneous boundary
data.
Indeed, given $f \in L^{2} (D,F)$ and
              $u_{0} \in B (\iG)$,
we look for a section $u \in \cD_{A}$ satisfying $Au = f$ in $D$ and
                                                 $t (u) = u_{0}$ on $\iG$.
By the very definition of the space $B (\iG)$ there is a $U_{0} \in \cD_{A}$
with the property that $\tau (U_{0}) = u_{0}$ on $\iG$.
This latter equality just amounts to saying that $U_{0} - u_{0} \in \cD_{T}$.
Set $u = U_{0} + U$, then $u \in \cD_{A}$ is equivalent to $U \in \cD_{A}$.
Furthermore, $t (u) = u_{0}$ on $\iG$ is equivalent to $t (U) = 0$.
Since $AU = f - A U_{0}$ and
      $A U_{0} \in L^{2} (D,F)$,
substituting $u = U_{0} + U$ into the problem leads to the Cauchy problem with
$u_{0} = 0$.

\begin{prob}
\label{pr.Cauchy.HD}
Let $f \in L^2 (D,F)$ be an arbitrary section.
Find $u \in \cD_{T}$ such that $Au = f$ in $D$.
\end{prob}

If $\iG^{\circ} \ne \emptyset$ and the Unique Continuation Property $(U)_{s}$
holds for $A$ then Problem \ref{pr.Cauchy.HD} has at most one solution,
   cf. Theorem 10.3.5 of \cite{Tark36}.
Otherwise we can not guarantee that the null-space
   $\Sol_{A} (D) \cap \cD_{T}$
of this problem is trivial.
It is well known that the Cauchy problem for elliptic equations is ill-posed
in general.
Moreover, if $A$ is overdetermined then additional necessary conditions arise
for the problem to be solvable.
In fact, these conditions reflect the fact that the image of $\cD_{T}$ by $A$
may be not dense in $L^2 (D,F)$.

Let us formulate this more precisely.
To this end, we invoke as usual the boundary conditions which are adjoint for
$t$ with respect to the Green formula in $D$.
Similarly to (\ref{eq.def.t}), for $g \in L^{2} (D,F)$ with
                                   $A^{\ast} g \in L^2 (D,E)$,
we say that $n (g) = 0$ on the set $\partial D \setminus \iG^{\circ}$
   if
\begin{equation}
\label{eq.def.n}
\int _{D} \left( (Au, g)_{x} - (u, A^{\ast} g)_{x} \right) dx = 0
\end{equation}
for all sections $u \in C^{\infty} (\overline{D}, E)$ satisfying $t (u) = 0$
on $\iG$.

Recall that
   $A^{1} \in \Diff^{m_{1}} (X; F, E^{2})$
stands for a compatibility operator for $A$ over $X$, i.e.,
   $A^{1}$ is in a sense ``smallest'' differential operator with the property
   that $A^{1} A \equiv 0$ on $X$.
We make use of the Green formula for $A^{1}$ in the same way as above to
introduce the relations
   ``$n (v) = 0$ on $\partial D \setminus \iG^{\circ}$'',
for all sections $v \in L^{2} (D,E^{2})$ with
                 $A^{1}{}^{\ast} v \in L^2 (D,F)$,
and
   ``$t (f) = 0$ on $\iG$'',
for all sections $f \in L^{2} (D,F)$ with
                 $A^{1} f \in L^2 (D,E^{2})$.

The boundary equations
   $n (v) = 0$ for sections of $E^{2}$ and
   $t (f) = 0$ for sections of $F$
are no longer induced by any Dirichlet system on $\partial D$ as those at
steps $1$ and
      $0$,
respectively.

\begin{lem}
\label{l.Necessary}
Assume that $f \in L^2 (D,F)$ belongs to the closure of $A\, \cD_{T}$ in
                                                        $L^2 (D,F)$.
Then
\begin{enumerate}
   \item [ $1)$ ]
$A^1 f = 0$ in $D$ in the sense of distributions;
   \item [ $2)$ ]
$t (f) = 0$ on $\iG$;
   \item [ $3)$ ]
$(f,g)_{L^2 (D,F)} = 0$ for all $g \in L^{2} (D,F)$ satisfying
   $A^{\ast} g = 0$ in $D$ and
   $n (g) = 0$ on $\partial D \setminus \iG^{\circ}$.
\end{enumerate}
\end{lem}

\begin{proof}
\
\\[-.3cm]

$1)$
Let $f$ belong to the closure of $A\, \cD_{T}$ in $L^2 (D,F)$.
Then there is a sequence $\{ u_k \}_{k \in \N}$ in $\cD_{T}$, such that
$\{ A u_k \}_{k \in \N} $ converges to $f$ in $L^2 (D,F)$.
Without loss of generality we may assume that each $u_{k}$ is of class
$C^{\infty} (\overline{D}, E)$, for such functions are dense in $\cD_{T}$.
As $A^1 A \equiv 0$, we get
\begin{eqnarray*}
   (f, A^1{}^{\ast} v)_{L^2 (D,F)}
 & = &
   \lim_{k \to \infty}
   (A u_k, A^1{}^{\ast} v)_{L^2 (D,F)}
                                                 \\
 & = &
   \lim_{k \to \infty}
   (u_k, (A^1 A)^{\ast} v)_{L^2 (D,E)}
                                                 \\
 & = &
   \lim_{k \to \infty}
   0
                                                 \\
 & = &
   0
\end{eqnarray*}
for all
   $v \in C^{\infty} (\overline{D}, E^2)$
satisfying $n (A^1{}^{\ast} v) = 0$ on $\partial D \setminus \iG^{\circ}$.
In particular, this equality is fulfilled for all sections
   $v \in C^{\infty} (\overline{D}, E^2)$
of compact supports in $D$, which implies $A^1 f = 0$ in $D$.

$2)$
Suppose $v \in C^{\infty} (\overline{D}, E^2)$ is any section satisfying
   $n (v) = 0$ on $\partial D \setminus \iG^{\circ}$.
Then
   $n (A^1{}^{\ast} v) = 0$ holds on $\partial D \setminus \iG^{\circ}$,
too, which is a consequence of $A^{\ast} A^1{}^{\ast} = 0$ and
                               Stokes' formula.
By $1)$, we get
\begin{eqnarray*}
   -\, (f, A^1{}^{\ast} v)_{L^2 (D,F)}
 & = &
   \int_{D}
   \left( (A^1 f, v)_x - (f, A^1{}^{\ast} v)_x \right)
   dx
                                                 \\
 & = &
   0,
\end{eqnarray*}
the first equality being a consequence of the fact that $A^{1} f = 0$ in $D$.
Hence it follows that $t (f) = 0$ on $\iG$.

$3)$
Finally,
\begin{eqnarray*}
   (f,g)_{L^2 (D,F)}
 & = &
   \lim_{k \to \infty}
   (A u_k, g)_{L^2 (D,F)}
                                                 \\
 & = &
   \lim_{k \to \infty}
   \int_{D} \left( (A u_k, g)_x - (u_k, A^{\ast} g)_x \right) dx
                                                 \\
 & = &
   \lim_{k \to \infty}
   0
                                                 \\
 & = &
   0
\end{eqnarray*}
provided that $g \in L^{2} (D,F)$ satisfies
   $A^{\ast} g = 0$ in $D$ and
   $n (g) = 0$ on $\partial D \setminus \iG^{\circ}$.
This proves $3)$.
\end{proof}

The condition $3)$ is not only necessary but also sufficient in order that $f$
would belong to the closure of $A\, \cD_{T}$ in $L^2 (D,F)$.

\begin{lem}
\label{l.Necessary.back.1}
If $f$ satisfies the condition $3)$ of Lemma \ref{l.Necessary} then $f$ lies
in the closure of $A\, \cD_{T}$ in $L^2 (D,F)$.
\end{lem}

\begin{proof}
Write $V$ for the space of all $g \in L^{2} (D,F)$ satisfying
   $A^{\ast} g = 0$ in $D$ and
   $n (g) = 0$ on $\partial D \setminus \iG^{\circ}$.
We shall have established the lemma if we show that $V$ coincides with the
orthogonal complement of the image $A\, \cD_{T}$ in $L^2 (D,F)$.
By definition,
   $g \in (A\, \cD_{T})^{\perp}$
if
\begin{equation}
\label{eq.ort.IMA}
(g, Au)_{L^2 (D,F)} = 0
\end{equation}
for all $u \in \cD_{T}$.
Since $\cD_{T}$ contains all smooth functions of compact support in $D$,
we conclude that $(A\, \cD_{T})^\bot \subset \Sol_{A^{\ast}} (D)$.
Then equality (\ref{eq.ort.IMA}) imlplies that $(A\, \cD_{T})^\bot \subset V$
because
$$
   (g, Au)_{L^2 (D,F)}
 = -
   \int_{D}
   \left( (A^* g, u)_x - (g, Au)_x \right)
   dx
$$
for all $g \in V$.

On the other hand, the inclusion $V \subset (A\, \cD_{T})^\bot$ follows from
(\ref{eq.def.t}) because each $u \in \cD_{T}$ can be approximated in the norm
$D (\cdot)$ by sections $u_{k} \in C^{\infty} (\overline{D}, E)$ satisfying
$t (u_{k}) = 0$ on $\iG$.
\end{proof}

Denote by $\cH^{1} (D,\iG)$ the space of all $g \in L^{2} (D,F)$ satisfying
   $A^{\ast} g = A^{1} g = 0$ in $D$
and
   $n (g) = 0$ on $\partial D \setminus \iG^{\circ}$.
Following \cite{SchuShlaTark2} we call $\cH^{1} (D,\iG)$ the harmonic space in
the Cauchy problem with data on $\iG$.
This is an analogue of the well-known harmonic spaces in the Neumann problem
for the Laplace operator,
   cf. \cite[4.1]{Tark35}.

\begin{lem}
\label{l.Necessary.back.2}
When combined with
\begin{enumerate}
   \item [ $4)$ ]
$(f,g)_{L^2 (D,F)} = 0$ for all $g \in \cH^{1} (D,\iG)$,
\end{enumerate}
the condition $1)$ of Lemma \ref{l.Necessary} implies that $f$ belongs to the
closure of $A\, \cD_{T}$ in $L^{2} (D,F)$.
\end{lem}

\begin{proof}
Let the conditions $1)$ and $4)$ are fulfilled for $f \in L^2 (D,F)$.
The proof of Lemma \ref{l.Necessary.back.1} shows that
\begin{equation}
\label{eq.decomp}
f = f_1 + f_2,
\end{equation}
where
   $f_1$ belongs to the closure of $A\, \cD_{T}$ in $L^{2} (D,F)$
and
   $f_2 \in V$.
As $A^1 f = 0$ in $D$, we deduce by Lemma \ref{l.Necessary} that $A^1 f_2 = 0$
in $D$.
This means $f_2 \in \cH^{1} (D,\iG)$.
Finally,  $4)$ implies
\begin{eqnarray*}
   0
 & = &
   (f, f_2)_{L^2 (D,F)}
                                                 \\
 & = &
   (f_2, f_2)_{L^2 (D,F)}
\end{eqnarray*}
whence $f_2 = 0$,
   and so $f$ belongs to the closure of $A\, \cD_{T}$ in $L^2 (D,F)$.
\end{proof}

Obviously, if $f$ belongs to the closure of $A\, \cD_{T}$ in $L^{2} (D,F)$
then it satisfies $4)$ by Lemma \ref{l.Necessary}, $3)$.
It follows that the condition $3)$ of Lemma \ref{l.Necessary} is equivalent to
$1) + 4)$.

\begin{lem}
\label{l.Necessary.back.3}
When combined with
\begin{enumerate}
   \item [ $5)$ ]
$(f,g)_{L^2 (D,F)} = 0$
   for all $g \in \cH^{1} (D,\iG)$ satisfying $t (g) = 0$ on $\iG$,
\end{enumerate}
the conditions $1)$ and $2)$ of Lemma \ref{l.Necessary} imply that $f$ belongs
to the closure of $A\, \cD_{T}$ in $L^{2} (D,F)$.
\end{lem}

\begin{proof}
Let the conditions $1)$, $2)$ and $5)$ hold true for $f \in L^2 (D,F)$.
Taking into account Lemma \ref{l.Necessary} and
                    decomposition (\ref{eq.decomp})
we readily conclude that $A^1 f_2 = 0$ in $D$ and
                         $t (f_2) = 0$ on $\iG$.
Finally, $5)$ implies
\begin{eqnarray*}
   0
 & = &
   (f, f_2)_{L^2 (D,F)}
                                                 \\
 & = &
   (f_2, f_2)_{L^2 (D,F)}
\end{eqnarray*}
whence $f_2 = 0$.
Thus, $f = f_{1}$ belongs to the closure of $A\, \cD_{T}$ in $L^2 (D,F)$,
   as desired.
\end{proof}

\begin{rem}
\label{r.elliptic}
Of course, if $A$ is elliptic then $A^1 = 0$ and the conditions $1)$ and
                                                                $2)$
are always fulfilled.
As for the condition $3)$, one easily proves that each $g \in L^{2} (D,F)$
satisfying
   $A^{\ast} g = 0$ in $D$ and
   $n (g) = 0$ on $\partial D \setminus \iG^{\circ}$
vanishes identically in all of $D$, provided that
   $A^{\ast}$ is elliptic,
   $\iG \ne \partial D$
and
   $A^{\ast}$ possesses the Unique Continuation Property $(U)_{s}$ in a
   neighbourhood of $\overline{D}$
(see, for instance, \cite[Theorem 10.3.5]{Tark36}).
If $A$ is overdetermined elliptic then the domain $D$ should possess some
convexity property relative to $A$, in order that
   $\cH^{1} (D,\iG)$ or
   $\{ g \in \cH^{1} (D,\iG) :\, t (g) = 0\ \mbox{on}\ \iG \}$
might be trivial.
In the case $\iG = \emptyset$ we refer the reader to \cite[4.1.3]{Tark35} for
more details.
\end{rem}

We have thus described the closure of $A\, \cD_{T}$ in $L^2 (D,F)$.
It is a more difficult task to describe the image $A\, \cD_{T}$ itself.
The following lemma is the first step in this direction.

\begin{lem}
\label{l.generalized}
Let $f \in L^2 (D,F)$ belong to the closure of $A\, \cD_{T}$ in $L^2 (D,F)$.
Then a section $u \in \cD_{T}$ is a solution to Problem \ref{pr.Cauchy.HD} if
and only if
\begin{equation}
\label{eq.generalized}
(Au, Av)_{L^2 (D,F)} = (f, Av)_{L^2 (D,F)}
\end{equation}
for all $v \in \cD_{T}$.
\end{lem}

\begin{proof}
If Problem \ref{pr.Cauchy.HD} is solvable and $u$ is one of its solutions then
(\ref{eq.generalized}) is obviously satisfied.

Conversely, if (\ref{eq.generalized}) holds for an element $u \in \cD_{T}$
then $A^{\ast} (Au-f) = 0$ in $D$ because the space $\cD_{T}$ contains all
smooth functions of compact support in $D$.
It follows that
\begin{eqnarray*}
   \int _{D} \left( (A^{\ast} (Au-f), v)_x - (Au-f, Av)_x \right) dx
 & = &
   - (Au-f, Av)_{L^2 (D,F)}
                                                 \\
 & = &
   0
\end{eqnarray*}
for all $v \in C^{\infty} (\overline{D}, E)$ satisfying $t (v) = 0$ on $\iG$,
which is due to (\ref{eq.generalized}).
Hence $n (Au-f) = 0$ on $\partial D \setminus \iG^{\circ}$.
Finally, since both $Au$ and $f$  belong to the closure of $A\, \cD_{T}$ in
$L^{2} (D,F)$, Lemma \ref{l.Necessary}, $3)$ shows that
$$
(Au-f, Au-f)_{L^2 (D,F)} = 0,
$$
i.e., $Au = f$ in $D$.
\end{proof}

In conclusion of this section let us clarify the meaning of
(\ref{eq.generalized}).
Namely, this equality amounts to saying that a solution $u \in \cD_{T}$ of the
Cauchy problem $Au = f$ is actually a solution to the mixed problem
\begin{equation}
\label{eq.mixed}
\left\{ \begin{array}{rclcl}
        A^{\ast} A u & = & A^{\ast} f & \mbox{in} & D;
                                                 \\
               t (u) & = & 0          & \mbox{on} & \iG,
                                                 \\
              n (Au) & = & n (f)      & \mbox{on} & \partial D \setminus
                                                    \iG^{\circ}.
        \end{array}
\right.
\end{equation}
Indeed, the proof of Lemma \ref{l.generalized} shows that
   $A^{\ast} Au = A^{\ast} f$ in $D$ in the sense of distributions
and
   $n (Au) = n (f)$ in the sense that $n (Au-f) = 0$ on
   $\partial D \setminus \iG^{\circ}$.
In particular, if
   $n (f)$ is well defined on $\partial D \setminus \iG^{\circ}$
then also $n(Au)$ is well defined on $\partial D \setminus \iG^{\circ}$.

Of course, the mixed problem (\ref{eq.mixed}) considered in appropriate spaces
gives nothing but (\ref{eq.generalized}).

In the next sections we will systematically use the generalised setting
(\ref{eq.generalized}) of Problem \ref{pr.Cauchy.HD} in order to derive its
solvability conditions.

\section{A perturbation}
\label{s.Ap}
\setcounter{equation}{0}

Equation (\ref{eq.generalized}) surprisingly shows that
Problem \ref{pr.Cauchy.HD} may be well posed in many cases.
Namely, this is the case if the Hermitian form
   $(A \cdot, A \cdot)_{L^2 (D,F)}$
is actually a scalar product on $\cD_{T}$ inducing the same topology as the
original scalar product $D (\cdot, \cdot)$.
For example, not only the gradient operator  $\nabla$ in $\R^n$ meets this
latter condition but also many other overdetermined elliptic operators $A$
with finite-dimensional kernel $\Sol_{A} (D)$.
Of course, $(A \cdot, A \cdot)_{L^2 (D,F)}$ is always a scalar product on
$\cD_{T}$ if $\iG \ne \emptyset$ and $A$ possesses the property $(U)_{s}$.
However, the completion of $\cD_{T}$ with respect to
   $(A \cdot,A \cdot)_{L^2 (D,F)}$
may lead to a space with elements of arbitrary order of growth near
$\partial D$.

This observation suggests us to perturb the Hermitian form
   $(A \cdot, A \cdot)_{L^2 (D,F)}$
thus obtaining a ``good'' scalar product on $\cD_{T}$.
For this purpose let us introduce a family of Hermitian forms
$$
   (u,v)_{\varepsilon}
 = (Au,Av)_{L^2 (D,F)} + \varepsilon\, (u,v)_{L^2 (D,E)}
$$
on $\cD_{T}$, parametrised by $\varepsilon > 0$.
For each fixed $\varepsilon > 0$, the corresponding norm
   $\| u \|_{\varepsilon} = \sqrt{(u,u)_{\varepsilon}}$
is equivalent to the graph norm  $D (u)$ on $\cD_{T}$.
More precisely, we get
\begin{equation}
\label{eq.norms}
   \min \{ 1, \sqrt{\varepsilon} \}\, D (u)
 \leq
   \| u \|_{\varepsilon}
 \leq
   \max \{ 1, \sqrt{\varepsilon} \}\,
   D (u)
\end{equation}
for all $u \in \cD_{A}$.

Taking into account Lemma \ref{l.generalized} we now consider the following
perturbed Cauchy problem:

\begin{prob}
\label{pr.Cauchy.perturbed}
Given any $f \in L^{2} (D,F)$ and
          $h \in L^{2} (D,E)$,
find an element $u_{\varepsilon} \in \cD_{T}$ satisfying
\begin{equation}
\label{eq.generalized.perturbed}
   (A u_\varepsilon, Av)_{L^2 (D,F)}
 + \varepsilon\, (u_{\varepsilon}, v)_{L^2 (D,E)}
 = (f, Av)_{L^2 (D,F)}
 + \varepsilon\, (h, v)_{L^2(D,E}
\end{equation}
for all $v \in \cD_{T}$.
\end{prob}

Note that the equation (\ref{eq.generalized.perturbed}) leads to a perturbation
of mixed problem (\ref{eq.mixed}), more precisely,
\begin{equation}
\label{eq.mixed.perturbed}
\left\{ \begin{array}{rclcl}
          A^{\ast} A u_{\varepsilon} + \varepsilon\, u_{\varepsilon}
        & =
        & A^{\ast} f + \varepsilon\, h
        & \mbox{in} & D;
                                                 \\
          t (u_{\varepsilon})
        & =
        & 0
        & \mbox{on}
        & \iG,
                                                 \\
          n (A u_{\varepsilon})
        & =
        & n (f)
        & \mbox{on}
        & \partial D \setminus \iG^{\circ}.
        \end{array}
\right.
\end{equation}

Indeed, since the space $\cD_{T}$ contains all smooth functions with compact
support in $D$, (\ref{eq.generalized.perturbed}) implies
   $A^{\ast} A u_\varepsilon + \varepsilon\, u_\varepsilon
  = A^{\ast} f + \varepsilon\, h$
in $D$ in the sense of distributions.
The boundary condition $t (u_\varepsilon) = 0$ on $\iG$ follows from
Lemma \ref{l.HD}.
Finally, $n (A u_\varepsilon) = n (f)$ holds in the sense that
   $n (A u_\varepsilon - f)$ on $\partial D \setminus \iG^{\circ}$
because
\begin{eqnarray*}
    A^{\ast} (A u_\varepsilon - f) & = &   \varepsilon (h - u_{\varepsilon})
                                                 \\
                                   & \in & L^2 (D,E)
\end{eqnarray*}
in $D$ and
\begin{eqnarray*}
\lefteqn{
   \int_{D}
   \left( (A^{\ast} (A u_\varepsilon - f), v)_x - (A u_\varepsilon - f, Av)_x
   \right)
   dx
}
                                                 \\
 & &
  =\
   \varepsilon\, (h - u_\varepsilon, v)_{L^2 (D,E)}
 - (A u_\varepsilon - f, Av)_{L^2 (D,F)}
                                                 \\[.1cm]
 & &
  =\
   0
\end{eqnarray*}
for all $v \in C^{\infty} (\overline{D}, E)$ satisfying $t (v) = 0$ on $\iG$,
the latter equality being due to (\ref{eq.generalized.perturbed}).
If the restriction of $n (f)$ to $\partial D\setminus \iG^{\circ}$ makes sense,
then the restriction of $n (Au)$ does so.

If considered in appropriate function spaces, the mixed problem
(\ref{eq.mixed.perturbed}) gives certainly nothing but
(\ref{eq.generalized.perturbed}).

In general, mixed problems (\ref{eq.mixed}) and
                           (\ref{eq.mixed.perturbed})
have non-coercive boundary conditions on $\partial D \setminus \iG^{\circ}$.
Hence they fail to be well-posed in the relevant weighted Sobolev spaces,
   cf. \cite{HaruSchuX}.
The principal difference between Problems \ref{pr.Cauchy.HD} and
                                         \ref{pr.Cauchy.perturbed}
is that the last one is well-posed in $\cD_{T}$.

\begin{lem}
\label{l.solvable.perturbed}
For every $\varepsilon > 0$,
          $f \in L^2 (D,F)$ and
          $h \in L^2 (D,E)$
there exists a unique solution
   $u_\varepsilon (f,h) \in \cD_{T}$
to Problem \ref{pr.Cauchy.perturbed}.
Moreover, it satisfies
$$
   \| u_\varepsilon (f,h) \|_{\varepsilon}
 \leq
   \| f \|_{L^2 (D,F)} + \sqrt{\varepsilon}\, \| h \|_{L^2 (D,E)}.
$$
\end{lem}

\begin{proof}
Really, the estimates (\ref{eq.norms}) imply that the vector space $\cD_{T}$
endowed with the scalar product $(\cdot, \cdot)_\varepsilon)$ is a Hilbert
space.
The Schwarz inequality yields
\begin{eqnarray*}
\lefteqn{
   \left| (f, Av)_{L^2 (D,F)} + \varepsilon\, (h,v)_{L^2 (D,E)} \right|
}
                                                 \\[.1cm]
 & &
 \leq\
   \| f \|_{L^2 (D,F)} \| Av \|_{L^2 (D,F)}
 + \varepsilon\, \| h \|_{L^2 (D,F)}\, \| v \|_{L^2 (D,E)}
                                                 \\
 & &
 \leq\
   \| f \|_{L^2 (D,F)} \| v \|_{\varepsilon}
 + \sqrt{\varepsilon}\,
   \| h \|_{L^2 (D,F)}\,
   \sqrt{\varepsilon \| v \|^2_{L^2 (D,E)}}
                                                 \\
 & &
 \leq\
   c_\varepsilon (f,h)\, \| v \|_\varepsilon
\end{eqnarray*}
with
$$
   c_\varepsilon (f,h)
 = \| f \|_{L^2 (D,F)} + \sqrt{\varepsilon}\, \| h \|_{L^2 (D,E)}.
$$
Hence the map
$$
v \mapsto (f,Av)_{L^2 (D,F)} + \varepsilon\, (h,v)_{L^2 (D,E)}
$$
defines a continuous linear functional $\cF_{f,h}$ on $\cD_{T}$, whose norm is
majorised by
   $\| \cF_{f,h} \| \leq c_\varepsilon (f,h)$.

We now use the Riesz theorem to conclude that there exists a unique element
   $u_\varepsilon (f,h) \in \cD_{T}$
with
$$
   \cF_{f,h} (v)
 = (u_\varepsilon (f,h), v)_{\varepsilon}
$$
for every $v \in \cD_{T}$.
Clearly,
   $u_\varepsilon (f,h)$ is a solution to Problem \ref{pr.Cauchy.perturbed}.
Finally, by the  Riesz theorem we get
$$
\| u_\varepsilon (f,h) \|_{\varepsilon} \leq c_\varepsilon (f,h),
$$
as desired.
\end{proof}

The equations (\ref{eq.mixed.perturbed}) show that
Lemma \ref{l.solvable.perturbed} gives information on the solvability of a
mixed problem for the elliptic operator $A^{\ast} A + \varepsilon$ with very
special data on $D$, $\iG$ and $\partial D \setminus \iG^{\circ}$.
Let us clarify what kind solvability theorems can be obtained for arbitrary
data.

For a triple $w \in L^{2} (D,E)$ and
\begin{equation}
\label{eq.data}
\begin{array}{rcl}
   u_0
 & \in
 & \oplus_{j=0}^{m-1} H^{2m-m_{j}-1/2} (\iG, F_j),
                                                 \\
   u_1
 & \in
 & \oplus_{j=0}^{m-1} H^{m-m_{j}-1/2} (\partial D \setminus \iG^{\circ}, F_j),
\end{array}
\end{equation}
we investigate the problem of finding a section $u$ of the bundle $E$ over $D$
which satisfies
\begin{equation}
\label{eq.mixed.good}
\left\{ \begin{array}{rclcl}
          A^{\ast} A u + \varepsilon\, u
        & =
        & w
        & \mbox{in} & D;
                                                 \\
          t (u)
        & =
        & u_0
        & \mbox{on}
        & \iG,
                                                 \\
          n (A u)
        & =
        & u_{1}
        & \mbox{on}
        & \partial D \setminus \iG^{\circ},
        \end{array}
\right.
\end{equation}
the equations in $D$ and on the boundary of $D$ being understood in a proper
sense.
From what has already been proved it is clear what we mean by this proper
sense, namely
\begin{eqnarray}
\label{eq.proper}
   (Au, g)_{L^2 (D,F)} - (u, A^{\ast} g)_{L^2 (D,E)}
 & = &
   (u_0, n (g))_{\oplus L^2 (\iG, F_j)},
                                                 \nonumber
                                                 \\
   (u,v)_{\varepsilon}
 & = &
   (w,v)_{L^2 (D,E)}
 - (u_1, t (v))_{\oplus L^2 (\partial D \setminus \iG^{\circ}, F_j)}
                                                 \nonumber
                                                 \\
\end{eqnarray}
for all
   $g \in C^{\infty} (\overline{D}, F)$
satisfying
   $n (g) = 0$ on $\partial D \setminus \iG^{\circ}$,
and for all
   $v \in C^{\infty} (\overline{D}, E)$
satisfying
   $t (v) = 0$ on $\iG$,
respectively.

\begin{thm}
\label{t.mixed.good}
Let $(A^{\ast} A)^2$ possess the Unique Continuation Property $(U)_{s}$.
Then, for every triple $(w, u_0, u_1)$ there exists a unique solution
   $u \in \cD_{A} \cap H^{2m}_{\loc} (D \cup \iG^{\circ}, E)$
to Problem \ref{eq.mixed.good}.
Moreover, there is a constant $C (\varepsilon) > 0$ which does not depend on
$(w, u_0, u_1)$, such that
\begin{equation}
\label{eq.mixed.good.estimate}
   \| u \|_\varepsilon^2
 \leq
   C (\varepsilon)
   \Big(
   \| w \|_{L^{2} (D,E)}^2
 + \| u_0 \|_{\oplus H^{2m-m_{j}-1/2} (\iG, F_j)}^2
 + \| u_1 \|_{\oplus H^{m-m_{j}-1/2} (\partial D \setminus \iG^{\circ}, F_j)}^2
   \Big).
\end{equation}
\end{thm}

\begin{proof}
Choose arbitrary $u_0$ and
                 $u_1$
as in (\ref{eq.data}).
Obviously, there are sections
$$
\begin{array}{rcl}
   U_0
 & \in
 & \oplus_{j=0}^{m-1} H^{2m-m_{j}-1/2} (\partial D, F_j),
                                                 \\
   U_1
 & \in
 & \oplus_{j=0}^{m-1} H^{m-m_{j}-1/2} (\partial D, F_j),
\end{array}
$$
such that
   $U_0 = u_0$ on $\iG$,
   $U_1 = u_1$ on $\partial D \setminus \iG^{\circ}$
and
\begin{eqnarray}
\label{eq.psi}
\lefteqn{
   \| U_0 \|_{\oplus H^{2m-m_{j}-1/2} (\partial D, F_j)}^2
 + \| U_1 \|_{\oplus H^{m-m_{j}-1/2} (\partial D, F_j)}^2
}
                                                 \nonumber
                                                 \\
 & &
 \leq
   2 \Big(
   \| u_0 \|_{\oplus H^{2m-m_{j}-1/2} (\iG, F_j)}^2
 + \| u_1 \|_{\oplus H^{m-m_{j}-1/2} (\partial D \setminus \iG^{\circ}, F_j)}^2
     \Big).
                                                 \nonumber
                                                 \\
\end{eqnarray}

As the pair $\{ t, n \circ A \}$ is a Dirichlet system of order $2m-1$ on
$\partial D$, solving the Dirichlet problem for $(A^{\ast} A)^2$ yields a
section $U' \in H^{2m} (D,E)$ with the following properties
\begin{equation}
\label{eq.Dirichlet.good}
\left\{ \begin{array}{rclcl}
          (A^{\ast} A)^{2}\, U'
        & =
        & 0
        & \mbox{in} & D;
                                                 \\
          t (U')
        & =
        & U_0
        & \mbox{on}
        & \partial D,
                                                 \\
          n (A U')
        & =
        & U_{1}
        & \mbox{on}
        & \partial D.
        \end{array}
\right.
\end{equation}
Moreover, there exists a positive constant $C > 0$ which is independent of $U$,
such that
\begin{equation}
\label{eq.psi2}
   \| U' \|_{H^{2m} (D,E)}^2
 \leq
   C
   \Big( \| U_0 \|_{\oplus H^{2m-m_{j}-1/2} (\partial D, F_j)}^2
       + \| U_1 \|_{\oplus H^{m-m_{j}-1/2} (\partial D, F_j)}^2
   \Big),
\end{equation}
see for instance \cite{Tark36}.

According to Lemma \ref{l.solvable.perturbed} there exists a solution
$U'' \in \cD_{T}$ to Problem \ref{pr.Cauchy.perturbed} with $f = 0$ and
\begin{eqnarray*}
   h
 & = &
   \frac{1}{\varepsilon} \left( w - A^{\ast} A U' \right) - U'
                                                 \\
 & \in &
   L^2 (D,E).
\end{eqnarray*}

Set $u = U' + U''$.
Then, integrating by parts and
      using Lemma \ref{l.solvable.perturbed}
we easily obtain
\begin{eqnarray*}
   (u,v)_{\varepsilon}
 & = &
   ((A^{\ast} A + \varepsilon) U', v)_{L^2 (D,E)}
 - (n (AU'), t (v))_{\oplus L^2 (\partial D, F_j)}
                                                 \\
 & + &
   (w,v)_{L^2 (D,E)}
 - ((A^{\ast} A + \varepsilon) U', v)_{L^2 (D,E)}
                                                 \\
 & = &
   (w, v)_{L^2 (D,E)}
 - (u_1, t (v))_{\oplus L^2 (\partial D \setminus \iG^{\circ}, F_j)}
\end{eqnarray*}
for every $v \in C^{\infty} (\overline{D}, E)$ satisfying $t (v) = 0$ on $\iG$,
i.e., the second equality of (\ref{eq.proper}) holds true.

On the other hand,
   for every $g \in C^{\infty} (\overline{D}, F)$ satisfying $n (g) = 0$ on
   $\partial D \setminus \iG^{\circ}$,
we get
\begin{eqnarray*}
\lefteqn{
(Au,g)_{L^2 (D,F)} - (u,A^{\ast} g)_{L^2 (D,E)}
}
                                                 \\
 & = &
   (A U',g)_{L^2 (D,F)} - (U',A^{\ast} g)_{L^2 (D,E)} +
   (A U'',g)_{L^2 (D,F)} - (U'',A^{\ast} g)_{L^2 (D,E)}
                                                 \\
 & = &
  (A U',g)_{L^2 (D,F)} - (U',A^{\ast} g)_{L^2 (D,E)}
\end{eqnarray*}
because $U'' \in \cD_{T}$.
Once again integrating by parts we obtain
\begin{eqnarray*}
   (AU',g)_{L^2 (D,F)} - (U',A^{\ast} g)_{L^2 (D,E)}
 & = &
   (t (U'), n (g))_{\oplus L^2 (\partial D,F_j)}
                                                 \\
 & = &
   (u_0, n (g))_{\oplus L^2 (\iG, F_j)},
\end{eqnarray*}
i.e., the first equality of (\ref{eq.proper}) is fulfilled.

By the elliptic regularity of the Dirichlet problem for the operator
$A^{\ast} A + \varepsilon$ we deduce that
   $u \in H^{2m}_{\loc} (D \cup \iG^{\circ}, E)$.

If all of $w$ and $u_{0}$, $u_{1}$ vanish then (\ref{eq.proper}) and
                                               Theorem \ref{t.HD}
imply that the corresponding solution $u$ lies in $\cD_{T}$.
On the other hand, the second equality of (\ref{eq.proper}) means that $u$ is
orthogonal to $\cD_{T}$ with respect to $(\cdot, \cdot)_{\varepsilon}$, i.e.,
$u \equiv 0$ which proves the uniqueness.

Finally, according to Lemma \ref{l.solvable.perturbed} we get
\begin{eqnarray*}
   \| u \|_{\varepsilon}
 \! & \! \leq \! & \!
   \| U' \|_\varepsilon + \| U'' \|_\varepsilon
                                                 \\
 \! & \! \leq \! & \!
   c\, \| U' \|_{H^{2m} (D,E)}
 + \frac{1}{\sqrt{\varepsilon}}
   \Big( \| w \|_{L^2 (D,E)} + \| A^{\ast} A U' \|_{L^2 (D,E)} \Big)
 + \sqrt{\varepsilon} \| U' \|_{L^2 (D,E)}.
\end{eqnarray*}
Combining this estimate with (\ref{eq.psi}) and
                             (\ref{eq.psi2})
we arrive at (\ref{eq.mixed.good.estimate}), as desired.
\end{proof}

One sees that the regularity up to $\partial D$ of the solution $u$ in
Theorem \ref{t.mixed.good} fails to correspond to the smoothness of the data
$w$ and $u_{0}$, $u_{1}$.
To justify this we recall that the boundary conditions $n \circ A$ on
$\partial D \setminus \iG^{\circ}$ are not coercive in general.
Were $n \circ A$ coercive we would have
   $u \in H^{2m}_{\loc} (\overline{D} \setminus \partial \iG, E)$.
However, we could not guarantee even in this case that $u \in H^{s} (D,E)$ for
some $s > m$ unless certain additional conditions were imposed on the triple
$(w,u_{0},u_{1})$ on $\partial \iG$.
This is typical for the mixed problems, cf. \cite{Eski1},
                                            \cite{HaruSchuX}
and elsewhere.

\section{The main theorem}
\label{s.Tmt}
\setcounter{equation}{0}

Set  $u_\varepsilon (f) = u_\varepsilon (f,0)$.
The inequalities (\ref{eq.norms}) and
                 Lemma \ref{l.solvable.perturbed}
give us a rough estimate for the family
   $\{ u_\varepsilon (f) \}_{\varepsilon > 0}$,
namely
$$
   D (u_\varepsilon (f))
 \leq
   \frac{1}{\sqrt{\varepsilon}}\, \| f \|_{L^2 (D,F)}.
$$
Thus, it might be unbounded while $\varepsilon \to 0+$.

Let us see how the behaviour of the family
   $\{ u_\varepsilon (f) \}_{\varepsilon > 0}$
reflects on the solvability of Problem \ref{pr.Cauchy.HD}.

\begin{thm}
\label{t.solvable.Cauchy}
The family $\{ u_\varepsilon (f) \}_{\varepsilon > 0}$ is bounded in $\cD_{T}$
if and only if there exists $u \in \cD_{T}$ satisfying (\ref{eq.generalized}).
\end{thm}

\begin{proof}
We first prove the following lemma.

\begin{lem}
\label{l.solvable.Cauchy.1}
Let there be a set
  $\Delta \subset (0,+\infty)$,
such that
\begin{enumerate}
   \item [ $1)$ ]
zero is an accumulation point of $\Delta$;
   \item [ $2)$ ]
the family $\{ u_\delta (f) \}_{\delta \in \Delta}$ is bounded in $\cD_{T}$.
\end{enumerate}
Then  there exists $u \in \cD_{T}$ satisfying (\ref{eq.generalized}).
\end{lem}

\begin{proof}
Suppose
   zero is an accumulation point of $\Delta$
and
   the family $\{ u_\delta (f) \}_{\delta \in \Delta}$ is bounded in $\cD_{T}$.
By (\ref{eq.generalized.perturbed}), we have
$$
   (A u_{\delta} (f), Av)_{L^2 (D,F)}
 + \delta\, (u_{\delta} (f), v)_{L^2 (D,E)}
 = (f, Av)_{L^2(D,F)}
$$
for all $v \in \cD_{T}$.
Passing to the limit, when $\Delta \ni \delta \to 0$, in the last equality
and using the fact that $\{ u_\delta (f) \}_{\delta \in \Delta}$ is bounded,
we obtain
\begin{equation}
\label{eq.delta.1}
   \lim_{\delta \to 0+} (A u_{\delta} (f), Av)_{L^2 (D,F)}
 = (f, Av)_{L^2 (D,F)}
\end{equation}
for all $v \in \cD_{T}$.

It is well known that every bounded set in a Hilbert  space is weakly compact.
Hence there is a subsequence
   $\{ u _{\delta_j} (f) \} \subset \cD_{T}$
weakly convergent in $\cD_{T}$ to an element $u \in \cD_{T}$.
Here, $\{ \delta_{j} \}$ converges to $0$ when $j \to \infty$.

Note that (\ref{eq.generalized.perturbed}) implies
$$
(u_\varepsilon (f), v)_{L^2 (D,E)} = 0
$$
for all $v \in \cD_{T} \cap \Sol_{A} (D)$, i.e., both
   $\{ u_{\delta_j} (f) \}$ and
   $u$
are $L^2 (D,E)\,$-orthogonal to $\cD_{T} \cap \Sol_{A} (D)$.
Let us show that
   $\{ u_{\delta_j} (f) \} $
converges weakly to $u$ in $L^2 (D,E)$ when $j \to \infty$.

Given any $v \in L^{2} (D,E)$, the map
   $u \mapsto (u,v)_{L^2 (D,E)}$
defines a continuous linear functional $\cF_v$ on $\cD_{T}$ with
   $\| \cF_v \| \leq \| v \|_{L^{2} (D,E)}$.
We now invoke the Riesz representation theorem to conclude that there exists a
unique element $\tilde{v} \in \cD_{T}$ with
   $D (u, \tilde{v}) = \cF_v (u)$
for every $u \in \cD_{T}$.
Hence
\begin{eqnarray*}
   \lim_{j \to \infty}
   (u_{\delta_j} (f), v)_{L^2 (D,E )}
 & = &
   \lim_{j \to \infty}
   D (u_{\delta_j} (f), \tilde{v})
                                                 \\
 & = &
   D (u, \tilde{v})
                                                 \\
 & = &
   (u, v)_{L^2 (D,E )}.
\end{eqnarray*}
This exactly means that
   $\{ u_{\delta_j} (f) \}$
converges weakly in $L^2 (D,E)$.

Now we easily calculate
\begin{eqnarray}
\label{eq.delta.2}
   \lim_{\Delta \ni \delta \to 0+}
   (A u_{\delta} (f), Av)_{L^2 (D,F)}
 & = &
   \lim_{\Delta \ni \delta \to 0+}
   \left( D (u_{\delta} (f), v) - (u_{\delta} (f), v)_{L^2 (D,E)} \right)
                                                 \nonumber
                                                 \\
 & = &
   D (u,v) - (u,v)_{L^2 (D,E)}
                                                 \nonumber
                                                 \\
 & = &
   (Au, Av)_{L^2 (D,F)}
                                                 \nonumber
                                                 \\
\end{eqnarray}
for all $v \in \cD_{T}$.
Combining (\ref{eq.delta.1}) and
          (\ref{eq.delta.2})
we see that  (\ref{eq.generalized}) holds true for $u$.
\end{proof}

Note that if (\ref{eq.generalized}) is solvable then there exists a solution
$u$ which is $L^2 (D,E)\,$-orthogonal to $\cD_{T} \cap \Sol_{A} (D)$.

We will have a stronger statement than Theorem \ref{t.solvable.Cauchy} if we
prove the following lemma.

\begin{lem}
\label{l.solvable.Cauchy.2}
If there exists $u \in \cD_{T}$ satisfying (\ref{eq.generalized}) then the
family
   $\{ u_\varepsilon (f) \}_{\varepsilon > 0}$
is bounded in $\cD_{T}$ and
$$
   \lim_{\varepsilon \to 0+} \| A (u_\varepsilon - u) \|_{L^2 (D,F)}
 = 0.
$$
Moreover,
   $\{ u_\varepsilon (f) \}_{\varepsilon > 0}$
converges weakly to $u \in \cD_{T}$ as  $\varepsilon \to 0+$,
if $u$ is $L^2 (D,E)\,$-orthogonal to $\cD_{T} \cap \Sol_{A} (D)$.
\end{lem}

\begin{proof}
Let there exist $u \in \cD_{T}$ satisfying (\ref{eq.generalized}).
Set $R_\varepsilon = u_\varepsilon (f) - u$.
Then (\ref{eq.generalized}) and
     (\ref{eq.generalized.perturbed})
imply
\begin{equation}
\label{eq.generalized.perturbed.2}
   (A R_\varepsilon, Av)_{L^2 (D,F)}
 + \varepsilon\, (R_\varepsilon, v)_{L^2 (D,E)}
 = - \varepsilon\, (u,v)_{L^2 (D,E)}
\end{equation}
for all $v \in \cD_{T}$, i.e.,
   $R_\varepsilon = u_\varepsilon (0, -u)$
is the solution to Problem \ref{pr.Cauchy.perturbed} with $f = 0$ and
                                                          $h = - u$.

According to (\ref{eq.norms}) and
             Lemma \ref{l.solvable.perturbed}
we have
\begin{eqnarray*}
   D (R_\varepsilon)
 & \leq &
   \frac{1}{\sqrt{\varepsilon}}\,
   \| R_\varepsilon \|_{\varepsilon}
                                                 \\
 & \leq &
   \frac{1}{\sqrt{\varepsilon}}\,
   \sqrt{\varepsilon}\,
   \| u \|_{L^2 (D,E)}
                                                 \\
 & = &
   \| u \|_{L^2 (D,E)}.
\end{eqnarray*}
Therefore, the family
   $\{ R_\varepsilon \}_{\varepsilon > 0}$
is bounded in $\cD_{T}$, and so the family
   $\{ u_\varepsilon (f) \}_{\varepsilon > 0}$
is bounded, too.
Now (\ref{eq.generalized.perturbed.2}) implies
\begin{eqnarray*}
   \lim_{\varepsilon \to 0+}
   \| A (u_\varepsilon (f) - u) \|_{L^2 (D,F)}^2
 & = &
   \lim_{\varepsilon \to 0+}
   \| A R_\varepsilon \|_{L^2 (D,F)}^2
                                                 \\
 & = &
 - \lim_{\varepsilon \to 0+}
   \varepsilon
   \left( \| R_\varepsilon \|^2_{L^2 (D,E)} + (u,R_\varepsilon)_{L^2 (D,E)}
   \right)
                                                 \\
 & = &
   0.
\end{eqnarray*}

Finally, let us prove that
   $\{ u_\varepsilon (f) \}_{\varepsilon > 0}$
converges weakly to $u$ in $\cD_{T}$ as $\varepsilon \to 0+$, provided that
$u$ is $L^2 (D,E)\,$-orthogonal to $\cD_{T} \cap \Sol_{A} (D)$.
We argue by contradiction.
Indeed, if
   $\{ u_\varepsilon (f) \}_{\varepsilon > 0}$
does not converge weakly to $u$ in $\cD_{T}$ then there are
   $v \in \cD_{T}$,
   $\gamma > 0$
and a sequence
   $\{ \varepsilon_j \}$ tending to $0 +$ as $j \to \infty$,
such that
\begin{equation}
\label{eq.contradiction}
   |D (u_{\varepsilon_j} - u, v)|
 \geq
   \gamma
\end{equation}
for every $j \in \N$.
But the sequence $\{ u_{\varepsilon_j} \}$ is bounded in the Hilbert space
$\cD_{T}$, and so it possesses a subsequence which converges weakly in
$\cD_{T}$.
By abuse of notation we denote it again by $\{ u_{\varepsilon_j} \}$.
As we have already seen in the proof of Lemma \ref{l.solvable.Cauchy.1}, the
weak limit of $\{ u_{\varepsilon_j} \}$ is $u$.
This contradicts (\ref{eq.contradiction}), and thus the assertion of the lemma
is proved.
\end{proof}

The proof of Theorem \ref{t.solvable.Cauchy} is complete.
\end{proof}

Note that if Problem \ref{pr.Cauchy.HD} is solvable then there exists a unique
solution $u$ which is $L^2 (D,E)\,$-orthogonal to $\cD_{T} \cap \Sol_{A} (D)$.

\begin{cor}
\label{c.solvable.Cauchy}
Suppose $f$ belongs to the closure of $A\, \cD_{T}$ in $L^{2} (D,F)$.
Then the family
   $\{ u_\varepsilon (f) \}_{\varepsilon > 0}$
is bounded in $\cD_{T}$ if and only if Problem \ref{pr.Cauchy.HD} is solvable.
Moreover,
$$
   \lim_{\varepsilon \to 0+}
   \| A u_\varepsilon (f) - f \|_{L^2 (D,F)}
 = 0
$$
and even
   $\{ u _\varepsilon (f) \}_{\varepsilon > 0}$
converges weakly, when $\varepsilon \to 0+$, to the solution $u \in \cD_{T}$
of Problem \ref{pr.Cauchy.HD} which is $L^2 (D,E)\,$-orthogonal to
$\cD_{T} \cap \Sol_{A} (D)$.
\end{cor}

\begin{proof}
This follows from Theorem \ref{t.solvable.Cauchy} and
                  Lemmas \ref{l.solvable.Cauchy.2} and
                         \ref{l.generalized}.
\end{proof}

Is it true that
   $\{ u_\varepsilon (f) \}_{\varepsilon > 0}$
converges to $u$ in the topology of $H^{m}_{\loc} (D \cup \iG^{\circ}, E)$
if $u \in \cD_{T}$ is the solution to Problem \ref{pr.Cauchy.HD} which is
$L^2 (D,E)\,$-orthogonal to $\cD_{T} \cap \Sol_{A} (D)$?
To answer this question we observe, by Lemma \ref{l.solvable.Cauchy.2}, that
the family
   $\{ u_\varepsilon (f) - u \}_{\varepsilon > 0}$
is bounded in $\cD_{T}$ and
\begin{eqnarray*}
   \lim_{\varepsilon \to 0+}
   \| A (u_\varepsilon (f) - u) \|_{L^2 (D,F)}
 & = &
   0,
                                                 \\
   t (u_\varepsilon (f) - u)
 & = &
   0
\end{eqnarray*}
on $\iG$ for every $\varepsilon > 0$.
Then, applying \cite[Theorem 7.2.6]{Tark37} we see that
   $\{ u_\varepsilon (f) \}_{\varepsilon > 0}$
converges to $u$ in $H^{m}_{\loc} (D \cup \iG^{\circ}, E)$.

\section{The well-posed case}
\label{s.6}
\setcounter{equation}{0}

It is well known that a linear operator $T : H \to \t{H}$ in normed spaces has
a continuous inverse if and only if
   $\| u \|_{H} \leq c\, \| Tu \|_{\t{H}}$
for every $u \in H$, the constant $c > 0$ being independent of $u$.
Hence, the (Cauchy) Problem \ref{pr.Cauchy.HD} is well-posed if and only if
there exists a constant $c > 0$ such that
\begin{equation}
\label{eq.well-posed}
\| u \|_{L^2 (D,E)} \leq c \| Au \|_{L^2 (D,F)}
\end{equation}
for all $u \in \cD_{T}$.

\begin{thm}
\label{t.well-posed.Cauchy}
Let the (Cauchy) Problem \ref{pr.Cauchy.HD} be well posed.
Then for every $f \in L^2 (D, F)$ there exists a limit
$$
u = \lim_{\varepsilon \to 0+} u_\varepsilon (f)
$$
in $\cD_{T}$.
Moreover, $u$ is the solution to Problem \ref{pr.Cauchy.HD} if $f$ belongs to
the closure of $A\, \cD_{T}$ in $L^{2} (D,F)$.
\end{thm}

\begin{proof}
Indeed, it follows from (\ref{eq.well-posed}) that the Hermitian form
$$
h (u,v) := (Au, Av)_{L^2 (D,F)}
$$
defines a scalar product on $\cD_{T}$ inducing the same topology as the
original one.
We now use the Riesz representation theorem to see that for every
$f \in L^2 (D,F)$ there is a unique element $u \in \cD_{T}$ satisfying
(\ref{eq.generalized}).

Moreover, (\ref{eq.well-posed}) yields
$$
   D (u_\varepsilon (f) - u)
 \leq
   \sqrt{c+1}\, \| u_\varepsilon (f) - u \|_{\varepsilon}.
$$
Then using (\ref{eq.generalized.perturbed.2}) and
           Lemma \ref{l.solvable.perturbed}
we see that
\begin{eqnarray*}
   D (u_\varepsilon (f) - u)
 & \leq &
   \sqrt{c+1}\, \| u_\varepsilon (0,-u) \|_{\varepsilon}
                                                 \\
 & \leq &
   \sqrt{c+1}\, \sqrt{\varepsilon}\, \| u \|_{L^{2} (D,E)}.
\end{eqnarray*}
Therefore, we get
$$
   \lim_{\varepsilon \to 0+}
   D (u_\varepsilon (f) - u)
 = 0,
$$
and so Corollary \ref{c.solvable.Cauchy} shows that $u$ is a solution to
Problem \ref{pr.Cauchy.HD} provided $f$ belongs to the closure of $A\, \cD_{T}$
in $L^{2} (D,F)$.
\end{proof}

Apparently, if $A$ is a differential operator with finite-dimensional kernel
$\Sol_{A} (D)$ then the (Cauchy) Problem \ref{pr.Cauchy.HD} is well posed for
$A$.

\begin{exmp}
\label{e.ODE}
Let
   $X = \R$,
   $A = d / dx$,
   $D = (a,b)$ with $-\infty < a < b <\infty$,
and
   $\iG = \{ a \}$.
Then $\cD_{A} = H^1 (D)$.
The Cauchy problem
$$
\left\{ \begin{array}{rclcl}
        u' (x) & = & f (x) & \mbox{for} & x \in (a,b),
                                                 \\
        u (a)  & = & u_0,  &            &
        \end{array}
\right.
$$
with $u_0 \in \R$, is known to be well posed in Sobolev spaces as well as in
spaces of smooth functions on $[a,b]$.
Its solution can be easily found by the formula
$$
u (x) = u_0 + \int_a^x f (y)\, dy.
$$
Let us look at the corresponding family of mixed problems.
In this case we have
   $A^{\ast} = - d / dx$ and
   $\partial D \setminus \iG^{\circ} = \{ b \}$,
hence the mixed problems are
$$
\left\{ \begin{array}{rclcl}
          u_{\varepsilon}'' (x) - \varepsilon\, u_{\varepsilon} (x)
        & =
        & f' (x)
        & \mbox{for}
        & x \in (a,b),
                                                 \\
          u_{\varepsilon} (a)
        & =
        & u_0,
        &
        &
                                                 \\
          u_{\varepsilon}' (b)
        & =
        & f (b),
        &
        &
        \end{array}
\right.
$$
where $u_{0} \in \R$ is arbitrary.
One easily calculates that
$$
   u_\varepsilon (x)
 = u_0
 + \int\limits_a^x \! f (y) \cosh (\sqrt{\varepsilon} (x-y))\, dy
 + \frac{\sinh (\sqrt{\varepsilon} (x-a))}{\cosh (\sqrt{\varepsilon} (b-a))}
   \int\limits_a^b \! f (y) \sinh (\sqrt{\varepsilon} (b-y))\, dy
$$
and
$$
\lim _{\varepsilon \to 0+} u_\varepsilon = u
$$
even in the norm of $C^1 [a,b]$, if $f \in C [a,b]$.
\end{exmp}

\section{Finding the solution}
\label{s.7}
\setcounter{equation}{0}

Let us discuss the very important question of how to find the solution of
Problem \ref{pr.Cauchy.perturbed}, and hence a solution to
Problem \ref{pr.Cauchy.HD}.
Of course, if an explicit orthonormal basis
   $\{ e_i \}_{i \in \N}$
in the space $\cD_{T}$ with the scalar product $(\cdot, \cdot)_\varepsilon$
is available, then one easily obtains
\begin{equation}
\label{eq.sol.mixed}
   u_\varepsilon (f,h)
 = \sum_{j=1}^\infty
   (u_\varepsilon (f,h), e_i)_{\varepsilon}\,
   e_i.
\end{equation}
According to (\ref{eq.generalized.perturbed}) we have
\begin{equation}
\label{eq.sol.coeff}
   (u_\varepsilon (f,h), e_i)_{\varepsilon}
 = (f, A e_i)_{L^2 (D,F)} + \varepsilon\, (h, e_i)_{L^2 (D,E)},
\end{equation}
hence (\ref{eq.sol.mixed}) and
      (\ref{eq.sol.coeff})
give us a complete description of the solution $u_\varepsilon (f,h)$ to
Problem \ref{pr.Cauchy.perturbed}.
Unfortunately, it is not an easy task to construct an explicit basis
   $\{ e_i \}_{i \in \N}$.

\begin{exmp}
\label{e.as}
Let
   $\iG = \partial D \cap S$
where $S$ is a sufficiently smooth hypersurface near $\partial D$.
Choose a defining function $\delta (x)$ for $S$.
Then we can start with a linearly independent system of the form
   $\{ (\delta (x))^{m-1} P_i (x) \}$
in $\cD_{T}$, where
   $P_{i} (x)$ are polynomials of increasing degree taking their values in
   $E_{x}$.
Orthogonalising it by the standard Gram-Schmidt procedure we arrive at an
orthonormal system in $\cD_{T}$.
In order to obtain a basis we have certainly to guarantee that the system
   $\{ (\delta (x))^{m-1} P_i (x) \}$
be dense in $\cD_{T}$.
However, for applications it suffices to have merely a finite number of basis
elements.
\end{exmp}

Let us describe an alternative way of finding the solution.
Assume that the operator $A^{\ast} A + \varepsilon$ possesses the Unique
Continuation property $(U)_{s}$ in a neighbourhood of $\overline{D}$.
Then it has a two-sided fundamental solution there
   (see for instance \cite{Tark35}).
Fix such a fundamental solution $\iPhi_\varepsilon (x,y)$ for
   $A^{\ast} A + \varepsilon$.
For each $s \geq 0$, it induces a continuous linear map
$
\iPhi_\varepsilon : H^{s} (D,E) \to H^{s+2m} (D,E)
$
by
$
u \mapsto r_{+}\, \iPhi_\varepsilon (e_{+} u)
$
where
   $e_{+}$ means the extension by zero to all of $X$ and
   $r_{+}$ the restriction to $D$.
This map actually extends to a continuous map
$
\iPhi_\varepsilon : H^{s} (D,E) \to H^{s+2m} (D,E)
$
for all $s \in \R$, being a right inverse of $A^{\ast} A + \varepsilon$.
Every element $u \in \cD_{A}$ may be thus written in the form
\begin{equation}
\label{eq.mixed.1}
u = U + \iPhi_\varepsilon ((A^{\ast} A + \varepsilon) u),
\end{equation}
where
   $U \in \cD_{A} \cap \Sol_{A^{\ast} A + \varepsilon} (D)$.
Indeed, fix $u \in \cD_{A}$.
Since $Au \in L^2 (D,F)$ we deduce that $A^{\ast} A u \in H^{-m} (D,E)$.
It follows that
\begin{eqnarray*}
   \iPhi_\varepsilon  ((A^{\ast} A + \varepsilon) u)
 & \in &
   H^m (D,E)
                                                 \\
 & \subset &
   \cD_{A}.
\end{eqnarray*}
Setting
   $U = u - \iPhi_\varepsilon  ((A^{\ast} A + \varepsilon) u)$
yields readily (\ref{eq.mixed.1}) with
   $U \in \cD_{A} \cap \Sol_{A^{\ast} A + \varepsilon} (D)$,
as desired.

In practice one usually has only a complete linearly independent system
   $\{ U_i \}_{i \in \N}$
of solutions to
   $(A^{\ast} A + \varepsilon) U = 0$
on neighborhoods of $\overline{D}$, or even on all of $X^{\circ}$.

\begin{lem}
\label{l.dense.HDA}
Assume that
   $A^{\ast} A + \varepsilon$
possesses the Unique Continuation Property $(U)_{s}$.
If
   $M \subset \Sol_{A^{\ast} A + \varepsilon} (\overline{D})$
is a dense set in
   $C^{m-1} (\overline{D}, E) \cap \Sol_{A^{\ast} A + \varepsilon} (D)$
then it is dense in
   $\cD_{A} \cap \Sol_{A^{\ast} A + \varepsilon} (D)$.
 \end{lem}

\begin{proof}
When endowed with the scalar product $(\cdot,\cdot)_{\varepsilon}$,
   $\cD_{A} \cap \Sol_{A^{\ast} A + \varepsilon} (D)$
is a Hilbert space.
Hence it suffices to prove that the orthogonal complement of $M$ in this space
is zero.

To this end, pick $u \in \cD_{A} \cap \Sol_{A^{\ast} A + \varepsilon} (D)$.
Since $u$ belongs to $L^2 (D,E)$ it has a finite order of growth near
$\partial D$,
   cf. \cite{Tark36}.
It follows that the expressions $t (u)$ and
                                $n (Au)$
have weak boundary values $u_0$ and
                          $u_1$
in the space of distributions on $\partial D$.

Let
   $v_{0} \in \oplus_{j=0}^{m-1} C^\infty (\partial D, F_j)$.
As $t$ is a Dirichlet system of order $m-1$ on $\partial D$, there is a section
$v \in C^\infty (\overline{D}, E)$ satisfying $t (v) = v_{0}$.
Then
$$
   <u_1, v_{0}>
 =:
   \lim_{\delta \to 0-}
   \int_{\partial D_{\delta}}
   \left( n (Au), v \right)_x\,  ds_\delta (x)
$$
and the definition does not depend on the particular choice of $v$.
Since the Dirichlet problem for $A^{\ast} A + \varepsilon$ in $D$ is uniquely
solvable over the whole scale of Sobolev spaces, we can take
   $v \in C^{\infty} (\overline{D}, E) \cap
          \Sol_{A^{\ast} A + \varepsilon} (D)$.

If $u$ is orthogonal to
   $M \subset \Sol_{A^{\ast} A + \varepsilon} (\overline D)$
with respect to the scalar product $(\cdot, \cdot)_{\varepsilon}$
   then
\begin{eqnarray*}
   0
 & = &
   (u, v)_\varepsilon
                                                 \\
 & = &
   \lim_{\delta \to 0-}
   \int_{D_{\delta}} (Au, Av)_x \, dx
 + \varepsilon\, (u, v)_{L^2 (D,E)}
                                                 \\
 & = &
   \lim_{\delta \to 0-}
   \Big( \int_{\partial D_{\delta}} (n (Au), t (v))_x\, d s_\delta (x)
       + \int_{D_{\delta}} (A^{\ast} A u, v)_x\, dx
   \Big)
 + \varepsilon\, (u, v)_{L^2 (D,E)}
                                                 \\
 & = &
   \lim_{\delta \to 0-}
   \int_{\partial D_{\delta}}
   (n (Au), t (v))_x\, ds_\delta (x)
\end{eqnarray*}
for all $v \in M$.
As $M$ is dense in
   $C^{m-1} (\overline{D}, E) \cap \Sol_{A^{\ast} A + \varepsilon} (D)$
it follows that $n (Au) =0$ on $\partial D$.

On the other hand, since $u \in \cD_{A}$ it can be approximated in the norm
$D (\cdot)$ by a sequence $\{ u_k \} \subset C^{\infty} (\overline{D}, E)$.
Then
\begin{eqnarray*}
   (u,u)_\varepsilon
 & = &
   \lim_{k \to \infty}
   (u,u_{k})_\varepsilon
                                                 \\
 & = &
   \lim_{k \to \infty}
   \lim_{\delta \to 0-}
   \int_{\partial D_{\delta}}
   (n (Au), u_{k})_x\, ds_\delta (x)
                                                 \\
 & = &
   0
\end{eqnarray*}
whence $u \equiv 0$ in $D$.
\end{proof}

For $M = \Sol_{A^{\ast} A + \varepsilon} (X^{\circ})$, the hypothesis of
Lemma \ref{l.dense.HDA} is not too restrictive.
It is fulfilled, e.g., if the complement of $D$ has no compact components in
$X^{\circ}$,
   see \cite{Tark35}.
In particular, this is the case if $\partial D$ is connected.

Applying to
   $\{ U_i \}_{i \in \N}$
the Gram-Schmidt orthogonalisation procedure with respect to the scalar
product $(\cdot, \cdot)_\varepsilon$, we obtain an orthonormal basis
   $\{ b_i = b_i (\varepsilon) \}_{i \in \N}$
in $\cD_{A} \cap \Sol_{A^{\ast} A + \varepsilon} (D)$.

The equality (\ref{eq.mixed.1}) suggest us to look for solutions to mixed
Problem \ref{pr.Cauchy.perturbed} of the form
\begin{equation}
\label{eq.mixed.2}
   u_\varepsilon (f,h)
 = \iPhi_\varepsilon (A^{\ast} f + \varepsilon h)
 + \sum_{i=1}^\infty c_i (\varepsilon) b_i (\varepsilon)
\end{equation}
where the series on the right-hand side converges in $\cD_{A}$.
The point is to find the coefficients $c_i (\varepsilon)$ through $f$ and
                                                                  $h$.
For this purpose, we denote by
   $\iPi_{\iG,\varepsilon}$
the orthogonal projection
$$
   \cD_{A} \cap \Sol_{A^{\ast} A + \varepsilon} (D)
 \to
   \cD_{T} \cap \Sol_{A^{\ast} A + \varepsilon} (D)
$$
with respect to the scalar product $(\cdot, \cdot)_{\varepsilon}$.

\begin{lem}
\label{l.mixed.2}
Each solution
   $u_\varepsilon (f,h) \in \cD_{T}$
of Problem \ref{pr.Cauchy.perturbed} may be written as the series
(\ref{eq.mixed.2}) where
$$
   c_i (\varepsilon)
 = (f, A \iPi_{\iG,\varepsilon} b_i)_{L^2 (D,F)}
 + \varepsilon\, (h, \iPi_{\iG,\varepsilon} b_i)_{L^2 (D,E)}
 - (\iPhi_\varepsilon  (A^{\ast} f + \varepsilon h), b_i)_\varepsilon.
$$
\end{lem}

\begin{proof}
Indeed, let
   $u_\varepsilon \in \cD_{T}$
be a solution of Problem \ref{pr.Cauchy.perturbed}.
As we have seen in \S \ref{s.Ap},
$$
   (A^{\ast} A + \varepsilon) u_\varepsilon
 = A^{\ast} f + \varepsilon h
$$
in $D$.
Using (\ref{eq.mixed.1}) we easily arrive at (\ref{eq.mixed.2}) with some
uniquely defined coefficients $c_i (\varepsilon)$.

Write
   $\tilde{\iPi}_{\iG,\varepsilon}$
for the orthogonal projection
   $\cD_{A} \to \cD_{T}$
with respect to $(\cdot, \cdot)_{\varepsilon}$.
Since $\tilde{\iPi}_{\iG,\varepsilon}$ is self-adjoint in $\cD_{A}$, we get
\begin{eqnarray}
\label{eq.Pi.e1}
   (u_\varepsilon, b_i)_\varepsilon
 & = &
   (\tilde{\iPi}_{\iG,\varepsilon} u_\varepsilon, b_i)_\varepsilon
                                                 \nonumber
                                                 \\
 & = &
   (u_\varepsilon, \tilde{\iPi}_{\iG,\varepsilon} b_i)_\varepsilon
                                                 \\
 & = &
   (f, A \tilde{\iPi}_{\iG,\varepsilon} b_i)_{L^2 (D,F)}
 + \varepsilon\, (h, \tilde{\iPi}_{\iG,\varepsilon} b_i)_{L^2 (D,E)},
                                                 \nonumber
\end{eqnarray}
the last equality being a consequence of (\ref{eq.generalized.perturbed}).

Now (\ref{eq.mixed.2}) implies
$$
   (u_\varepsilon, b_i)_\varepsilon
 = (\iPhi_\varepsilon  (A^{\ast} f + \varepsilon h), b_i)_\varepsilon
 + c_i (\varepsilon).
$$
Combining this with (\ref{eq.Pi.e1}) yields
$$
   c_i (\varepsilon)
 = (f, A \tilde{\iPi}_{\iG,\varepsilon} b_i)_{L^2 (D,F)}
 + \varepsilon\, (h, \tilde{\iPi}_{\iG,\varepsilon} b_i)_{L^2 (D,E)}
 - (\iPhi_\varepsilon (A^{\ast} f + \varepsilon h), b_i)_\varepsilon.
$$
Finally, for every $v \in C^{\infty}_{\comp} (D,E)$ we get
\begin{eqnarray*}
   (\tilde{\iPi}_{\iG,\varepsilon} b_i, (A^{\ast} A + \varepsilon) v
   )_{L^2 (D,E)}
 & = &
   (A \tilde{\iPi}_{\iG,\varepsilon} b_i, Av)_{L^2 (D,F)}
 + \varepsilon\, (\tilde{\iPi}_{\iG,\varepsilon} b_i, v)_{L^2 (D,E)}
                                                 \\
 & = &
   (b_i, \tilde{\iPi}_{\iG,\varepsilon} v)_{\varepsilon}
                                                 \\
 & = &
   (b_i, v)_{\varepsilon}
                                                 \\
 & = &
   ((A^{\ast} A + \varepsilon) b_i, v)_{L^2 (D,E)}
                                                 \\
 & = &
   0.
\end{eqnarray*}
This means that
   $\tilde{\iPi}_{\iG,\varepsilon} b_i$
belongs to
   $\cD_{T} \cap \Sol_{A^{\ast} A + \varepsilon} (D)$
whence
   $\tilde{\iPi}_{\iG,\varepsilon} b_i = \iPi_{\iG,\varepsilon} b_i$,
showing the lemma.
\end{proof}

We have thus derived expressions for the coefficients $c_i (\varepsilon)$
through $f$ and $h$.
However, it is not an easy task to explicitly construct the family of
projections $\{ \iPi_{\iG,\varepsilon} \}$.

\begin{lem}
\label{l.proj}
For every
   $u \in \cD_{A} \cap \Sol_{A^{\ast} A + \varepsilon} (D)$,
the projection $\iPi_{\iG,\varepsilon} u$ just amounts to  the solution of
Problem \ref{pr.Cauchy.perturbed} with $f = Au$ and
                                       $h = u$ .
\end{lem}

\begin{proof}
By the very definition,
$
   \iPi_{\iG,\varepsilon} u
 \in
   \cD_{T} \cap \Sol_{A^{\ast} A + \varepsilon} (D)
$
and
$$
   (u - \iPi_{\iG,\varepsilon} u, v)_{\varepsilon}
 = 0
$$
for all $v \in \cD_{T}$ satisfying $(A^{\ast} A + \varepsilon) v = 0$ in $D$.

Further, the solution
   $u_\varepsilon = u_\varepsilon (Au,u)$
of Problem \ref{pr.Cauchy.perturbed} with $f = Au$ and
                                          $h = u$
belongs to $\cD_{T} \cap \Sol_{A^{\ast} A + \varepsilon} (D)$ because
$
   A^{\ast} f  + \varepsilon h
 = (A^{\ast} A + \varepsilon) u
 = 0.
$
Moreover, (\ref{eq.mixed.perturbed}) gives
$$
   (u - u_{\varepsilon}, v)_{\varepsilon}
 = 0
$$
for all $v \in \cD_{T}$.

We wish to show that
   $\iPi_{\iG,\varepsilon} u = u_\varepsilon$,
which is equivalent to
$
\| \iPi_{\iG,\varepsilon} u - u_\varepsilon \|_{\varepsilon} = 0.
$
To this end, write
\begin{eqnarray*}
   (\iPi_{\iG,\varepsilon} u - u_\varepsilon,
    \iPi_{\iG,\varepsilon} u - u_\varepsilon)_{\varepsilon}
 \! & \! = \! & \!
   -\,
   ((u - \iPi_{\iG,\varepsilon} u) - (u - u_\varepsilon),
    \iPi_{\iG,\varepsilon} u - u_\varepsilon)_{\varepsilon}
                                                 \\
 \! & \! = \! & \!
   -\,
   (u - \iPi_{\iG,\varepsilon} u,
    \iPi_{\iG,\varepsilon} u - u_\varepsilon)_{\varepsilon}
   -
   (u - u_\varepsilon,
    \iPi_{\iG,\varepsilon} u - u_\varepsilon)_{\varepsilon}.
\end{eqnarray*}
By the above, both summands on the right-hand side vanish because
   $\iPi_{\iG,\varepsilon} u - u_\varepsilon$
belongs to
   $\cD_{T} \cap \Sol_{A^{\ast} A + \varepsilon} (D)$.
\end{proof}

Of course, the lemma does not allow one to effectively determine the Fourier
coefficients $c_i$.
On the one hand, to find $c_i$ we only need to know
   $\iPi_{\iG,\varepsilon} b_i$.
On the other hand, this requires, by Lemma \ref{l.proj}, a solution of
   Problem \ref{pr.Cauchy.perturbed}
with very special data $f$ and $h$.

Let us now describe how to find solutions to Problem \ref{pr.Cauchy.perturbed}
for ``good'' data.
For this purpose we introduce for $s \geq 2m$ the Hermitian form
$$
h (u,v)
 = (t (u), t (v))_{\oplus
                   H^{s-m_{j}-1/2} (\iG, F_j)}
 + (n (Au), n (Av))_{\oplus
                     H^{s-m-m_{j}-1/2} (\partial D \setminus \iG^{\circ}, F_j)}
$$
on the space $H$ of all
   $u \in \cD_{A} \cap \Sol_{A^{\ast} A + \varepsilon} (D)$
with the property that
$$
\begin{array}{rcl}
   t (u)
 & \in
 & \oplus_{j=0}^{m-1}
   H^{s-m_{j}-1/2} (\iG, F_j),
                                                 \\
   n (Au)
 & \in
 & \oplus_{j=0}^{m-1}
   H^{s-m-m_{j}-1/2} (\partial D \setminus \iG^{\circ}, F_j),
\end{array}
$$
the expressions $t (u)$ and $n (Au)$ being understood in the sense of weak
boundary values.

\begin{lem}
\label{l.HGamma.Hilbert}
Suppose $s \geq 2m$.
When endowed with the scalar product $h (\cdot,\cdot)$, $H$ is a Hilbert
space.
\end{lem}

\begin{proof}
Indeed, (\ref{eq.mixed.good.estimate}) implies that $h (\cdot,\cdot)$ is a
scalar product on $H$.
Moreover, if $\{ u_k \}$ is a Cauchy sequence in $H$ then it a Cauchy sequence
in $\cD_{A} \cap \Sol_{A^{\ast} A + \varepsilon} (D)$.
Since this latter is a Hilbert space, $\{ u_k \}$ has a limit $u$ in this
space.
Moreover, both $\{ t (u_{k}) \}$ and
               $\{ n (Au_{k}) \}$
converge to $t (u)$ and
            $n (Au)$
in the space of distributions on $\partial D$, or, more precisely, in
   $\oplus_{j=0}^{m-1} H^{-m_{j}-1/2} (\partial D, F_j)$ and
   $\oplus_{j=0}^{m-1} H^{-m-m_{j}-1/2} (\partial D, F_j)$,
respectively.
By assumption,
   $\{ t (u_k) \}$ and
   $\{ n(Au_k) \}$
are Cauchy sequences in the Hilbert spaces
   $\oplus_{j=0}^{m-1} H^{s-m_{j}-1/2} (\iG, F_j)$ and
   $\oplus_{j=0}^{m-1}
    H^{s-m-m_{j}-1/2} (\partial D \setminus \iG^{\circ}, F_j)$,
respectively.
Hence, they converge to elements $u_{0}$ and
                                 $u_{1}$
in these spaces.
Finally, the uniqueness of a limit yields
   $t (u) = u_{0}$ on $\iG$ and
   $n (Au) = u_{1}$ on $\partial D \setminus \iG^{\circ}$,
i.e., $u \in H$, which completes the proof.
\end{proof}

Let $\{ U_i \}_{i \in \N}$ be a complete linearly independent system in $H$.
Applying the Gram-Schmidt orthogonalisation to $\{ U_i \}_{i \in \N}$ we get
an orthonormal basis $\{ B_i \}_{i \in \N}$ in $H$.

\begin{thm}
\label{c.mixed.good}
Let $s \geq 2m$.
Then, for every
   $w \in H^{s-2m} (D,E)$
and
$$
\begin{array}{rcl}
   u_{0}
 & \in
 & \oplus_{j=0}^{m-1}
   H^{s-m_{j}-1/2} (\iG, F_j),
                                                 \\
   u_{1}
 & \in
 & \oplus_{j=0}^{m-1}
   H^{s-m-m_{j}-1/2} (\partial D \setminus \iG^{\circ}, F_j),
\end{array}
$$
the series
$$
u = \iPhi_\varepsilon (w) + \sum_{i=1}^\infty k_i B_i
$$
converges in $\cD_{A}$ and satisfies (\ref{eq.mixed.good}),
   provided that
$$
   k_i
 = h (u - \iPhi_\varepsilon (w), B_{i}).
$$
\end{thm}

\begin{proof}
This is a direct consequence of Theorem \ref{t.mixed.good}.
Recall that the boundary equations
                       $t (u) = u_{0}$ on $\iG$ and
                       $n (Au) = u_{1}$ on $\partial D \setminus \iG^{\circ}$
are interpreted in the sense of (\ref{eq.proper}).
\end{proof}

From this theorem we deduce, in particular, that
$$
   \iPi_{\iG,\varepsilon} b_i
 = \sum_{q=1}^\infty k_i B_i,
$$
with the coefficients
$$
   k_i
 = (n (A b_i)), n (A B_i)
   )_{\oplus H^{s-m-m_{j}-1/2} (\partial D \setminus \iG^{\circ}, F_j)}.
$$

Knowing $\iPi_{\iG,\varepsilon} b_i$ we can find, by Lemma \ref{l.mixed.2},
the solution of Problem \ref{pr.Cauchy.perturbed} for any data
   $f \in L^2 (D,F)$ and
   $h \in L^2 (D,F)$.
Of course, if both $f$ and $h$ are smooth enough, namely $f \in H^{s-m} (D,F)$
                                                     and $h \in H^{s-2m} (D,E)$
with $s \geq 2m$, then we can determine the solution of
Problem \ref{pr.Cauchy.perturbed} directly by Theorem \ref{c.mixed.good}.

One question still unanswered is whether a complete system
   $\{ U_i \}_{i \in \N}$ in $H$
may be chosen to consists of solutions to $(A^{\ast} A + \varepsilon) u = 0$
on neighbourhoods of $\overline{D}$.
Analysis similar to that in the proof of Lemma \ref{l.dense.HDA} shows that
this is always the case if $\partial D$ is smooth enough, e.g., of class
$C^{2m-1}$.

\section{Dirac operators}
\label{s.Do}
\setcounter{equation}{0}

Let $X = \R^n$, where $n\geq 2$, and $E = \R^n \times \C^k$,
                                     $F = \R^n \times \C^l$.
The sections of $E$ are functions of $n$ real variables with values in
$\C^{k}$, and similarly for $F$.

Let $A$  be a Dirac operator, i.e., a homogeneous first order differential
operator with constant coefficients in $\R^{n}$,
$$
A = \sum_{j=1}^n A_j\, \frac{\partial}{\partial x_j},
$$
such that
\begin{equation}
\label{eq.Dirac}
(\sigma (A) (\xi))^{\ast} \sigma (A) (\xi) = |\xi|^2\, E_{k}
\end{equation}
for all $\xi \in \R^n$.
Here,
   $A_j$ are $(l \times k)\,$-matrices of complex numbers
and
   $E_{k}$ is the identity $(k \times k)\,$-matrix.

The Dirac operators satisfy $A^{\ast} A =  - E_{k}\, \Delta$, where $\Delta$
is the usual Laplace operator in $\R^n$.

The perturbed mixed problem (\ref{eq.mixed.perturbed}) reads as
$$
\left\{ \begin{array}{rclcl}
          (- \Delta + \varepsilon) u_{\varepsilon}
        & = & A^{\ast} f
        & \mbox{in}
        & D;
                                                 \\
          t (u_{\varepsilon})
        & =
        & 0
        & \mbox{on}
        & \iG,
                                                 \\
          n (A u_{\varepsilon})
        & =
        & n (f)
        & \mbox{on}
        & \partial D \setminus \iG^{\circ},
        \end{array}
\right.
$$
where
$$
n (f) = (\sigma (A) (\nabla \varrho))^{\ast} f
$$
and
   $\varrho$ is a defining function of $D$ in the sense of (\ref{eq.rho}).
Thus, this is a family of mixed problems for the Helmholtz equation.

We are going to study the (Cauchy) Problem \ref{pr.Cauchy.HD} on the unit ball
$D = \mathbb{B}$ in $\R^n$.
To this end, we pass to spherical coordinates
   $x = r\, S (\varphi)$
where $\varphi$ are coordinates on the unit sphere $\partial D = \mathbb{S}$
in $\R^{n}$.
The Laplace operator $\Delta$ in the spherical coordinates takes the form
\begin{equation}
\label{eq.Laplace.spher}
   \Delta
 = \frac{1}{r^2}
   \Big( \Big( r \frac{\partial}{\partial r} \Big)^2
       + (n-2) \Big( r \frac{\partial}{\partial r} \Big)
       - \Delta_{\mathbb{S}}
   \Big),
\end{equation}
where $\Delta_{\mathbb{S}}$ is the Laplace-Beltrami operator on the unit
sphere.

To solve the homogeneous equation
   $(- \Delta + \varepsilon) u_{\varepsilon} = 0$
we make use of the Fourier method of separation of variables.
Writing
   $u_{\varepsilon} (r,\varphi) = g (r,\varepsilon) h (\varphi)$
we get two separate equations for $g$ and
                                  $h$,
namely
\begin{eqnarray*}
   \Big( \Big( r \frac{\partial}{\partial r} \Big)^2
       + (n-2) \Big( r \frac{\partial}{\partial r} \Big)
       - \varepsilon r^{2}
   \Big) g
 & = &
   c\, g
                                                 \\
   \Delta_{\mathbb{S}} h
 & = &
   c\, h,
\end{eqnarray*}
$c$ being an arbitrary constant.

The second equation has non-zero solutions if and only if $c$ is an eigenvalue
of $\Delta_{\mathbb{S}}$.
These are well known to be $c = i (n + i - 2)$, for $i = 0, 1, \ldots$
   (see for instance \cite{TikhSamaX}).
The corresponding eigenfunctions of $\Delta_{\mathbb{S}}$ are spherical
harmonics $h_i (\varphi)$ of degree $i$, i.e.,
\begin{equation}
\label{eq.spher.harm}
   \Delta_{\mathbb{S}} h_i
 = i (n + i - 2)\, h_i.
\end{equation}

Consider now the following ordinary differential equation with respect to the
variable $r > 0$
\begin{equation}
\label{eq.Bessel}
   \Big( \Big( r \frac{\partial}{\partial r} \Big)^2
       + (n-2) \Big( r \frac{\partial}{\partial r} \Big)
       - \left( i (n + i - 2) + \varepsilon r^2 \right)
   \Big) g (r, \varepsilon)
 = 0.
\end{equation}
This is a version of the Bessel equation, and the space of its solutions is
two-dimensional.

For example, if $\varepsilon = 0$ then
   $g (r,0) = a r^i + b r^{2-i-n}$
with arbitrary constants $a$ and
                         $b$
is a general solution to (\ref{eq.Bessel}).
In this situation any function $r^i h_i (\varphi)$ is a homogeneous harmonic
polynomial.
In the general case the space of solutions to (\ref{eq.Bessel}) contains
a one-dimensional subspace of functions bounded at the point $r = 0$, cf.
\cite{TikhSamaX}.

For $i = 0, 1, \ldots$, fix a non-zero solution $g_i (r,\varepsilon)$ of
(\ref{eq.Bessel}) which is bounded at $r = 0$.
Then
\begin{equation}
\label{eq.Helmholtz.q}
   \left( - \Delta + \varepsilon \right)
   \left( g_i (r,\varepsilon) h_i (\varphi) \right)
 = 0
\end{equation}
on all of $\R^n$.
Indeed, by (\ref{eq.Laplace.spher}),
           (\ref{eq.spher.harm}) and
           (\ref{eq.Bessel})
we conclude that this equality holds in $\R^n \setminus \{ 0 \}$.
We now use the fact that $g_i (r,\varepsilon) h_i (\varphi)$ is bounded at the
origin to see that (\ref{eq.Helmholtz.q}) holds.

It is known that there are exactly
$$
J (i) = \frac{(n + 2i - 2)(n + i - 3)!}{i! (n-2)!}
$$
linearly independent spherical harmonics of degree $i$.
In \cite{Shla3} a system
$$
\{ H_{i}^{(j)} (\varphi) \}_{i = 0, 1, \ldots \atop
                             j = 1, \ldots, k\, J (i)}
$$
of $\C^{k}\,$-valued functions is constructed,
   such that
\begin{enumerate}
   \item [ $1)$ ]
the components of $H_{i}^{(j)} (\varphi)$ are spherical harmonics of degree
$i$;
   \item [ $2)$ ]
$\{ H_{i}^{(j)} (\varphi) \}$ is an orthonormal basis in
$L^2 (\mathbb{S}, E)$;
   \item [ $3)$ ]
$\{ A\, (r^i H_{i}^{(j)} (\varphi)) \}$ is an orthogonal system in
$L^2 (\mathbb{B}, F)$.
\end{enumerate}
More precisely, this system
   $\{ H_{i}^{(j)} (\varphi) \}$
consists of eigenfunctions of the operator $n \circ A$,
\begin{equation}
\label{eq.eigen.nA}
   (\sigma (A) (r S (\varphi)))^{\ast}
   A \left( r^i  H_{i}^{(j)} (\varphi) \right)
 =
   \lambda_{i}^{(j)}\, \left( r^i  H_{i}^{(j)} (\varphi) \right),
\end{equation}
where $\lambda_i^{(j)} \geq 0$.

\begin{lem}
\label{l.ort}
The system
$$
\{ b_{i}^{(j)} (r, \varphi, \varepsilon)
   := g_i (r,\varepsilon)\, H_{i}^{(j)} (\varphi)
\}_{i = 0, 1, \ldots \atop
    j = 1, \ldots, k\, J (i)}
$$
is orthogonal with respect to both Hermitian forms
   $(\cdot, \cdot)_{L^2 (\mathbb{B}, E)}$ and
   $(A \cdot, A \cdot)_{L^2 (\mathbb{B}, F)}$.
\end{lem}

\begin{proof}
Indeed, as
   $\{ H_{i}^{(j)} \}$
is an orthonormal basis in the space $L^2 (\mathbb{S}, E)$ on the unit
sphere, the system
   $\{ b_{i}^{(j)} \}$
is orthogonal in $L^2 (\mathbb{B}, E)$ because
\begin{eqnarray*}
   (b_{i}^{(j)}, b_{p}^{(q)})_{L^2 (\mathbb{B}, E)}
 & = &
   (H_{i}^{(j)}, H_{p}^{(q)})_{L^2 (\mathbb{S}, E)}
   \int_0^1 r^{n-1} g_i (r,\varepsilon) \overline{g_p (r,\varepsilon)}\, dr                                                 \\
 & = &
   0
\end{eqnarray*}
for $i \ne p$ or
    $j \ne q$.

Further, integrating by parts we get
\begin{equation}
\label{eq.ort.A1}
   ( A b_{i}^{(j)}, A b_{p}^{(q)} )_{L^2 (\mathbb{B}, F)}
 =
 - ( b_{i}^{(j)}, \Delta b_{p}^{(q)} )_{L^2 (\mathbb{B}, E)}
 + g_i (1,\varepsilon)\,
   ( H_{i}^{(j)}, n (A b_{p}^{(q)}) )_{L^2 (\mathbb{S}, E)}.
\end{equation}
On the other hand, (\ref{eq.Helmholtz.q}) implies
\begin{equation}
\label{eq.ort.A2}
 -\,
   ( b_{i}^{(j)}, \Delta b_{p}^{(q)} )_{L^2 (\mathbb{B}, E)}
 + \varepsilon\, ( b_{i}^{(j)}, b_{p}^{(q)} )_{L^2 (\mathbb{B}, E)}
 = 0
\end{equation}
for $i \ne p$ or
    $j \ne q$.

Let us write the expression $n \circ A$ in spherical coordinates.
Denote by $S' (\varphi)$ the Jacobi matrix of $S (\varphi)$.
Set
$$
   \left( S' (\varphi) \right)^{-1}
 :=
   \left( \left( S' (\varphi) \right)^T S' (\varphi)
   \right)^{-1}
   \left( S' (\varphi) \right)^T.
$$
Since the rank of $S' (\varphi)$ is equal to $n-1$, the inverse matrix of
$\left( S' (\varphi) \right)^T S' (\varphi)$ exists and is smooth.
Moreover,
   $\left( S' (\varphi) \right)^{-1}$
is a left inverse for $S' (\varphi)$.
An easy calculation shows that
$$
   \frac{\partial}{\partial x_j}
 = S_j (\varphi)\, \frac{\partial}{\partial r}
 + \frac{1}{r}\,
   \sum_{i=1}^{n-1}
   \left( S' (\varphi) \right)^{-1}_{i,j} \frac{\partial}{\partial \varphi_i}
$$
where $\left( S' (\varphi) \right)^{-1}_{i,j}$ is the $(i,j)\,$-entry of
      $\left( S' (\varphi) \right)^{-1}$.

Now (\ref{eq.Dirac}) implies
\begin{equation}
\label{eq.nA.spher}
   n \circ A
 =
   \sum_{k=1}^n A_k^{\ast}\, r S_k (\varphi)\,
   \sum_{j=1}^n A_j\, \frac{\partial}{\partial x_j}
 =
   r \frac{\partial}{\partial r} + R (\varphi, \partial_{\varphi})
\end{equation}
where
$$
   R (\varphi, \partial_{\varphi})
 = \sum_{k=1}^n A_k^{\ast}\, S_k (\varphi)\,
   \sum_{j=1}^n A_j\,
   \sum_{i=1}^{n-1}
   \left( S' (\varphi) \right)^{-1}_{i,j} \frac{\partial}{\partial \varphi_i}.
$$
Using (\ref{eq.eigen.nA}) and
      (\ref{eq.nA.spher})
we conclude that
\begin{eqnarray*}
   \lambda_{i}^{(j)} (r^i H_{i}^{(j)} (\varphi))
 & = &
   n (A (r^i H_{i}^{(j)} (\varphi)))
                                                 \nonumber
                                                 \\
 & = &
   i\, r^i\,   H^{(j)}_i (\varphi)
 + r^i\, R (\varphi, \partial_{\varphi}) H^{(j)}_i (\varphi).
\end{eqnarray*}
Hence
$$
   R (\varphi, \partial_{\varphi}) H^{(j)}_i (\varphi)
 = \left( \lambda_{i}^{(j)} - i \right) H_{i}^{(j)} (\varphi),
$$
and so (\ref{eq.nA.spher}) yields
\begin{eqnarray}
\label{eq.nA.Bessel}
   n (A b_{i}^{(j)})
 & = &
   r\, g_{i}^{\prime}\,  H_{i}^{(j)}
 + g _i\,  R (\varphi, \partial_{\varphi}) H_{i}^{(j)}
                                                 \\
 & = &
   \left( r g_{i}^{\prime} + (\lambda_{i}^{(j)} - i) g_i \right) H_{i}^{(j)}.
                                                 \nonumber
\end{eqnarray}
Therefore,
\begin{equation}
\label{eq.ort.A3}
   ( H_{i}^{(j)}, n (A b_{p}^{(q)}) )_{L^2 (\mathbb{S}, E)}
 = 0
\end{equation}
for $i \ne p$ or $j \ne q$.

Combining (\ref{eq.ort.A1})
          (\ref{eq.ort.A2}) and
          (\ref{eq.ort.A3})
we see that the system $\{ b_{i}^{(j)} \}$ is orthogonal with respect to
$(A \cdot, A \cdot)_{L^2 (\mathbb{B}, F)}$.
\end{proof}

\begin{rem}
\label{r.non-zero}
Note that
$
   g_i^{\prime} (1,\varepsilon) + (\lambda_{i}^{(j)} - i) g_i (1,\varepsilon)
 \ne 0
$
for all $\varepsilon > 0$.
Indeed, otherwise $n (A b_{i}^{(j)}) = 0$ on $\mathbb{S}$ and
   (\ref{eq.ort.A1}),
   (\ref{eq.ort.A2})
would imply $b_{i}^{(j)} \equiv 0$, which is wrong.
\end{rem}

\begin{thm}
\label{p.basis.gamma}
For every $\delta > 0$, the system
$$
   \{ b_{i}^{(j)} (r, \varphi, \varepsilon) \}_{i = 0, 1, \ldots \atop
                                                j = 1, \ldots, k\, J (i)}
$$
is an orthogonal basis in the space
   $\cD_{A} \cap \Sol_{- \Delta + \varepsilon E_{k}} (\mathbb{B})$
with the scalar product $(\cdot, \cdot)_{\delta}$.
\end{thm}

\begin{proof}
The orthogonality follows immediately from Lemma \ref{l.ort}.
As for the  completeness of the system
   $\{ b_{i}^{(j)} \}$
in $\cD_{A} \cap \Sol_{- \Delta + \varepsilon E_{k}} (\mathbb{B})$, we observe
that the estimates (\ref{eq.norms}) guarantee that every scalar product
   $(\cdot,\cdot)_\delta$
with $\delta > 0$ induces in $\cD_{A}$ the same topology as $D (\cdot,\cdot)$.
Hence it is sufficient to prove the completeness for $\delta = 1$.
Finally, since the system of harmonics $\{ H_i^{(j)} \}$ is dense in
                                       $C^{m-1} (\mathbb{S}, E)$
we see that $\{ b_i^{(j)} \}$ is dense in
   $C^{m-1} (\mathbb{S}, E) \cap
    \Sol_{- \Delta + \varepsilon E_{k}} (\mathbb{B})$.
Then the completeness is a consequence of Lemma \ref{l.dense.HDA}.
\end{proof}

As a fundamental solution $\iPhi_\varepsilon (x,y)$ of the operator
   $- \Delta + \varepsilon$
in $\R^3$ we may choose one of the standard kernels
$$
\iPhi_\varepsilon (x,y) = e^{\pm \sqrt{\varepsilon} |x-y|}.
$$
In $\R^2$ we can take as $\iPhi_\varepsilon (x,y)$ a Hankel function,
   see for instance \cite{TikhSamaX}.

\begin{exmp}
\label{eq.d}
Let $A = \nabla$ be the gradient operator in $\R^{n}$.
For every domain $D \subset \subset \R^n$, we have $\cD_{A} = H^{1} (D)$.
Since the estimate (\ref{eq.well-posed}) holds true for $\nabla$
   (see \cite{Mikh(ailov)X}),
the (Cauchy) Problem \ref{pr.Cauchy.HD} is well posed in $\cD_{T}$.
In this case $k = 1$,
             $l = n$,
   $A^{\ast} = - \mathrm{div}$ is a multiple of the divergence operator in
   $\R^n$
and
$$
   n \circ A
 = |x|\, \frac{\partial}{\partial n}
 = r\, \frac{\partial}{\partial r}
$$
where $\partial / \partial n$ is the derivative along the outward unit normal
vector to $\partial D$.
In particular, this means that every homogeneous harmonic polynomial $r^i h_i$
is an eigenfunction of $n \circ A$ corresponding to the eigenvalue
$\lambda_i = i$.
For example, in $\R^2$ we can take
\begin{eqnarray*}
   b_0^{(1)}
 & = &
   \frac{1}{\sqrt{2 \pi}}\, g_0 (r,\varepsilon),
                                                 \\
   b_i^{(1)}
 & = &
   \frac{1}{\sqrt{\pi}}\, g_i (r,\varepsilon) \cos (i \varphi),
                                                 \\
   b_i^{(2)}
 & = &
   \frac{1}{\sqrt{\pi}}\, g_i (r,\varepsilon) \sin (i \varphi),
\end{eqnarray*}
where $g_i$ are Hankel's functions.
In the case $s = 5/2$ and
            $\iG = \{ r=1,\,  \varphi \in [0,\pi] \}$
the Gram-Schmidt orthogonalisation in $H$ gives
\begin{eqnarray*}
   B_0^{(1)}
 & = &
   \frac{g_0 (r,\varepsilon)}
        {\sqrt{\pi} \sqrt{|g_0 (1,\varepsilon)|^2
                        + |g_0^{\prime} (1,\varepsilon)|^2}},
                                                 \\
   B_1^{(1)}
 & = &
   \frac{2\, g_1 (r,\varepsilon) \cos \varphi}
        {\sqrt{\pi} \sqrt{|g_1 (1,\varepsilon)|^2
                        + |g_1^{\prime} (1,\varepsilon)|^2}},
                                                 \\
   B_1^{(2)}
 & = &
   \frac{2 a g_0 (r,\varepsilon) + \sqrt{\pi} g_1 (r,\varepsilon) \sin \varphi}
        {\sqrt{b}},
\end{eqnarray*}
with
\begin{eqnarray*}
   a
 & = &
   g_0 (1,\varepsilon) g_1 (1,\varepsilon)
 - g_0^{\prime} (1,\varepsilon) g_1^{\prime} (1,\varepsilon),
                                                 \\
   b
 & = &
   \frac{3}{2}\, \pi^{2}
 + 4 a (1 + |g_0 (1,\varepsilon)|^2 + |g_0^{\prime} (1,\varepsilon)|^2),
\end{eqnarray*}
and so on.
\end{exmp}

\begin{exmp} \label{eq.CR}
Let
   $A := \partial_{1} + \sqrt{-1} \partial_{2}$
be ($2\,$-multiple of) the Cauchy-Riemann operator in $\C$.
Then the (Cauchy) Problem \ref{pr.Cauchy.HD} is ill-posed in $\cD_{T}$.
In this case $k = l = 1$,
             $A^{\ast} = - \partial_{1} + \sqrt{-1} \partial_{2}$
and
$$
   n \circ A
 = \bar{z}\, \bar{\partial}
 = r\, \frac{\partial}{\partial r}
 + \sqrt{-1}\, \frac{\partial}{\partial \varphi}
$$
hold.
The system $\{ b_i^{(j)} \}$ may be chosen as follows
\begin{eqnarray*}
   b_0^{(1)}
 & = &
   \frac{1}{\sqrt{2 \pi}}\, g_0 (r,\varepsilon),
                                                 \\
   b_i^{(1)}
 & = &
   \frac{1}{\sqrt{\pi}}\, g_i (r,\varepsilon) e^{\sqrt{-1}\, i \varphi},
                                                 \\
   b_i^{(2)}
 & = &
   \frac{1}{\sqrt{\pi}}\, g_i (r,\varepsilon) e^{-\sqrt{-1}\, i \varphi},
\end{eqnarray*}
with $\lambda_0^{(1)} = 0$,
     $\lambda_i^{(1)} = 0$ and
     $\lambda_i^{(2)} = 2 i$.
\end{exmp}

{\it Acknowledgements\,}
This article was written during the stay of the first author at the Institute
of Mathematics, University of Potsdam.
He gratefully acknowledges the financial support of
   the Deutscher Akademischer Austauschdienst and
   the RFBR grant 02--01--00167.

\newpage


\end{document}